\documentclass[11pt]{amsart}
\usepackage{amsthm}
\usepackage{amsmath}
\usepackage{amssymb}
\usepackage{verbatim}
\usepackage{mathrsfs}
\usepackage{eucal}
\let\mathcal=\CMcal

\allowdisplaybreaks[4]

\begin{document}
\def\sbt{\raisebox{1.2pt}{$\scriptscriptstyle\,\bullet\,$}}

\def\alp{\alpha}
\def\bet{\beta}
\def\gam{\gamma}
\def\del{\delta}
\def\eps{\epsilon}
\def\zet{\zeta}
\def\tht{\theta}
\def\iot{\iota}
\def\kap{\kappa}
\def\lam{\lambda}
\def\sig{\sigma}
\def\ome{\omega}
\def\vep{\varepsilon}
\def\vth{\vartheta}
\def\vpi{\varpi}
\def\vrh{\varrho}
\def\vsi{\varsigma}
\def\vph{\varphi}
\def\Gam{\Gamma}
\def\Del{\Delta}
\def\Tht{\Theta}
\def\Lam{\Lambda}
\def\Sig{\Sigma}
\def\Ups{\Upsilon}
\def\Ome{\Omega}
\def\vka{\varkappa}
\def\vDe{\varDelta}
\def\vSi{\varSigma}
\def\vTh{\varTheta}
\def\vGm{\varGamma}
\def\vOm{\varOmega}
\def\vPi{\varPi}
\def\vPh{\varPhi}
\def\vPs{\varPsi}
\def\vUp{\varUpsilon}
\def\vXi{\varXi}

\def\frka{{\mathfrak a}}    \def\frkA{{\mathfrak A}}
\def\frkb{{\mathfrak b}}    \def\frkB{{\mathfrak B}}
\def\frkc{{\mathfrak c}}    \def\frkC{{\mathfrak C}}
\def\frkd{{\mathfrak d}}    \def\frkD{{\mathfrak D}}
\def\frke{{\mathfrak e}}    \def\frkE{{\mathfrak E}}
\def\frkf{{\mathfrak f}}    \def\frkF{{\mathfrak F}}
\def\frkg{{\mathfrak g}}    \def\frkG{{\mathfrak G}}
\def\frkh{{\mathfrak h}}    \def\frkH{{\mathfrak H}}
\def\frki{{\mathfrak i}}    \def\frkI{{\mathfrak I}}
\def\frkj{{\mathfrak j}}    \def\frkJ{{\mathfrak J}}
\def\frkk{{\mathfrak k}}    \def\frkK{{\mathfrak K}}
\def\frkl{{\mathfrak l}}    \def\frkL{{\mathfrak L}}
\def\frkm{{\mathfrak m}}    \def\frkM{{\mathfrak M}}
\def\frkn{{\mathfrak n}}    \def\frkN{{\mathfrak N}}
\def\frko{{\mathfrak o}}    \def\frkO{{\mathfrak O}}
\def\frkp{{\mathfrak p}}    \def\frkP{{\mathfrak P}}
\def\frkq{{\mathfrak q}}    \def\frkQ{{\mathfrak Q}}
\def\frkr{{\mathfrak r}}    \def\frkR{{\mathfrak R}}
\def\frks{{\mathfrak s}}    \def\frkS{{\mathfrak S}}
\def\frkt{{\mathfrak t}}    \def\frkT{{\mathfrak T}}
\def\frku{{\mathfrak u}}    \def\frkU{{\mathfrak U}}
\def\frkv{{\mathfrak v}}    \def\frkV{{\mathfrak V}}
\def\frkw{{\mathfrak w}}    \def\frkW{{\mathfrak W}}
\def\frkx{{\mathfrak x}}    \def\frkX{{\mathfrak X}}
\def\frky{{\mathfrak y}}    \def\frkY{{\mathfrak Y}}
\def\frkz{{\mathfrak z}}    \def\frkZ{{\mathfrak Z}}

\def\cal{\fam2}
\def\cala{{\cal A}}
\def\calb{{\cal B}}
\def\calc{{\cal C}}
\def\cald{{\cal D}}
\def\cale{{\cal E}}
\def\calf{{\cal F}}
\def\calg{{\cal G}}
\def\calh{{\cal H}}
\def\cali{{\cal I}}
\def\calj{{\cal J}}
\def\calk{{\cal K}}
\def\call{{\cal L}}
\def\calm{{\cal M}}
\def\caln{{\cal N}}
\def\calo{{\cal O}}
\def\calp{{\cal P}}
\def\calq{{\cal Q}}
\def\calr{{\cal R}}
\def\cals{{\cal S}}
\def\calt{{\cal T}}
\def\calu{{\cal U}}
\def\calv{{\cal V}}
\def\calw{{\cal W}}
\def\calx{{\cal X}}
\def\caly{{\cal Y}}
\def\calz{{\cal Z}}

\def\AA{{\mathbb A}}
\def\BB{{\mathbb B}}
\def\CC{{\mathbb C}}
\def\DD{{\mathbb D}}
\def\EE{{\mathbb E}}
\def\FF{{\mathbb F}}
\def\GG{{\mathbb G}}
\def\HH{{\mathbb H}}
\def\II{{\mathbb I}}
\def\JJ{{\mathbb J}}
\def\KK{{\mathbb K}}
\def\LL{{\mathbb L}}
\def\MM{{\mathbb M}}
\def\NN{{\mathbb N}}
\def\OO{{\mathbb O}}
\def\PP{{\mathbb P}}
\def\QQ{{\mathbb Q}}
\def\RR{{\mathbb R}}
\def\SS{{\mathbb S}}
\def\TT{{\mathbb T}}
\def\UU{{\mathbb U}}
\def\VV{{\mathbb V}}
\def\WW{{\mathbb W}}
\def\XX{{\mathbb X}}
\def\YY{{\mathbb Y}}
\def\ZZ{{\mathbb Z}}

\def\bfa{{\mathbf a}}    \def\bfA{{\mathbf A}}
\def\bfb{{\mathbf b}}    \def\bfB{{\mathbf B}}
\def\bfc{{\mathbf c}}    \def\bfC{{\mathbf C}}
\def\bfd{{\mathbf d}}    \def\bfD{{\mathbf D}}
\def\bfe{{\mathbf e}}    \def\bfE{{\mathbf E}}
\def\bff{{\mathbf f}}    \def\bfF{{\mathbf F}}
\def\bfg{{\mathbf g}}    \def\bfG{{\mathbf G}}
\def\bfh{{\mathbf h}}    \def\bfH{{\mathbf H}}
\def\bfi{{\mathbf i}}    \def\bfI{{\mathbf I}}
\def\bfj{{\mathbf j}}    \def\bfJ{{\mathbf J}}
\def\bfk{{\mathbf k}}    \def\bfK{{\mathbf K}}
\def\bfl{{\mathbf l}}    \def\bfL{{\mathbf L}}
\def\bfm{{\mathbf m}}    \def\bfM{{\mathbf M}}
\def\bfn{{\mathbf n}}    \def\bfN{{\mathbf N}}
\def\bfo{{\mathbf o}}    \def\bfO{{\mathbf O}}
\def\bfp{{\mathbf p}}    \def\bfP{{\mathbf P}}
\def\bfq{{\mathbf q}}    \def\bfQ{{\mathbf Q}}
\def\bfr{{\mathbf r}}    \def\bfR{{\mathbf R}}
\def\bfs{{\mathbf s}}    \def\bfS{{\mathbf S}}
\def\bft{{\mathbf t}}    \def\bfT{{\mathbf T}}
\def\bfu{{\mathbf u}}    \def\bfU{{\mathbf U}}
\def\bfv{{\mathbf v}}    \def\bfV{{\mathbf V}}
\def\bfw{{\mathbf w}}    \def\bfW{{\mathbf W}}
\def\bfx{{\mathbf x}}    \def\bfX{{\mathbf X}}
\def\bfy{{\mathbf y}}    \def\bfY{{\mathbf Y}}
\def\bfz{{\mathbf z}}    \def\bfZ{{\mathbf Z}}

\def\scra{{\mathscr A}}
\def\scrb{{\mathscr B}}
\def\scrc{{\mathscr C}}
\def\scrd{{\mathscr D}}
\def\scre{{\mathscr E}}
\def\scrf{{\mathscr F}}
\def\scrg{{\mathscr G}}
\def\scrh{{\mathscr H}}
\def\scri{{\mathscr I}}
\def\scrj{{\mathscr J}}
\def\scrk{{\mathscr K}}
\def\scrl{{\mathscr L}}
\def\scrm{{\mathscr M}}
\def\scrn{{\mathscr N}}
\def\scro{{\mathscr O}}
\def\scrp{{\mathscr P}}
\def\scrq{{\mathscr Q}}
\def\scrr{{\mathscr R}}
\def\scrs{{\mathscr S}}
\def\scrt{{\mathscr T}}
\def\scru{{\mathscr U}}
\def\scrv{{\mathscr V}}
\def\scrw{{\mathscr W}}
\def\scrx{{\mathscr X}}
\def\scry{{\mathscr Y}}
\def\scrz{{\mathscr Z}}

\def\phm{\phantom}
\def\smallstrut{\vphantom{\vrule height 3pt }}
\def\bdm #1#2#3#4{\left(
\begin{array} {c|c}{\ds{#1}}
 & {\ds{#2}} \\ \hline
{\ds{#3}\vphantom{\ds{#3}^1}} &  {\ds{#4}}
\end{array}
\right)}

\def\GL{\mathrm{GL}}
\def\PGL{\mathrm{PGL}}
\def\SL{\mathrm{SL}}
\def\Mp{\mathrm{Mp}}
\def\SU{\mathrm{SU}}
\def\SO{\mathrm{SO}}
\def\O{\mathrm{O}}
\def\U{\mathrm{U}}
\def\Mat{\mathrm{M}}
\def\Tr{\mathrm{Tr}}
\def\tr{\mathrm{tr}}
\def\ch{\mathrm{ch}}
\def\Nr{\mathrm{Nrd}}
\def\Ad{\mathrm{Ad}}
\def\ch{\mathrm{ch}}
\def\Her{\mathrm{Her}}
\def\Paf{\mathrm{Paf}}
\def\Pf{\mathrm{Pf}}
\def\Gal{\mathrm{Gal}}
\def\Sch{\mathrm{Sch}}
\def\Spec{\mathrm{Spec}\,}
\def\new{\mathrm{new}}
\def\Wh{\mathrm{Wh}}
\def\FJ{\mathrm{FJ}}
\def\Fj{\mathrm{Fj}}
\def\Sym{\mathrm{Sym}}
\def\Aut{\mathrm{Aut}}
\def\Dif{\mathrm{Diff}}
\def\GK{\mathrm{GK}}
\def\EGK{\mathrm{EGK}}
\def\supp{\mathrm{supp}}
\def\proj{\mathrm{proj}}
\def\vol{\mathrm{vol}}
\def\Hom{\mathrm{Hom}}
\def\End{\mathrm{End}}
\def\Ker{\mathrm{Ker}}
\def\Res{\mathrm{Res}}
\def\res{\mathrm{res}}
\def\cusp{\mathrm{cusp}}
\def\Irr{\mathrm{Irr}}
\def\rank{\mathrm{rank}}
\def\sgn{\mathrm{sgn}}
\def\diag{\mathrm{diag}}
\def\nd{\mathrm{nd}}
\def\lw{\mathrm{lw}}
\def\d{\mathrm{d}}
\def\ur{\mathrm{ur}}
\def\La{\langle}
\def\Ra{\rangle}
\newcommand{\lra}{\longrightarrow}
\newcommand{\lla}{\longleftarrow}

\def\trs{\,^t\!}
\def\tri{\,^\iot\!}
\def\iu{\sqrt{-1}}
\def\oo{\hbox{\bf 0}}
\def\ono{\hbox{\bf 1}}
\def\smallcirc{\lower .3em \hbox{\rm\char'27}\!}
\def\AAf{{\AA_\bff}}
\def\thalf{\tfrac{1}{2}}
\def\bsl{\backslash}
\def\wtl{\widetilde}
\def\til{\tilde}
\def\Ind{\operatorname{Ind}}
\def\ind{\operatorname{ind}}
\def\cind{\operatorname{c-ind}}
\def\ord{\operatorname{ord}}
\def\beq{\begin{equation}}
\def\eeq{\end{equation}}
\def\d{\mathrm{d}}
\newcommand{\bF}{\overline{\mathbb{F}}}

\newcounter{one}
\setcounter{one}{1}
\newcounter{two}
\setcounter{two}{2}
\newcounter{thr}
\setcounter{thr}{3}
\newcounter{fou}
\setcounter{fou}{4}
\newcounter{fiv}
\setcounter{fiv}{5}
\newcounter{six}
\setcounter{six}{6}
\newcounter{sev}
\setcounter{sev}{7}

\newcommand{\shp}{\rm\char'43}

\def\lddots{\mathinner{\mskip1mu\raise1pt\vbox{\kern7pt\hbox{.}}\mskip2mu\raise4pt\hbox{.}\mskip2mu\raise7pt\hbox{.}\mskip1mu}}
\newcommand{\1}{1\hspace{-0.25em}{\rm{l}}}

\makeatletter
\def\varddots{\mathinner{\mkern1mu
    \raise\p@\hbox{.}\mkern2mu\raise4\p@\hbox{.}\mkern2mu
    \raise7\p@\vbox{\kern7\p@\hbox{.}}\mkern1mu}}
\makeatother

\def\today{\ifcase\month\or
 January\or February\or March\or April\or May\or June\or
 July\or August\or September\or October\or November\or December\fi
 \space\number\day, \number\year}

\makeatletter
\def\varddots{\mathinner{\mkern1mu
    \raise\p@\hbox{.}\mkern2mu\raise4\p@\hbox{.}\mkern2mu
    \raise7\p@\vbox{\kern7\p@\hbox{.}}\mkern1mu}}
\makeatother

\def\today{\ifcase\month\or
 January\or February\or March\or April\or May\or June\or
 July\or August\or September\or October\or November\or December\fi
 \space\number\day, \number\year}

\makeatletter
\@addtoreset{equation}{section}
\def\theequation{\thesection.\arabic{equation}}

\theoremstyle{plain}
\newtheorem{theorem}{Theorem}[section]
\newtheorem*{main_proposition}{Proposition \ref{prop:51}}
\newtheorem*{main_theorem}{Theorem \ref{thm:51}}
\newtheorem{lemma}[theorem]{Lemma}
\newtheorem{proposition}[theorem]{Proposition}
\theoremstyle{definition}
\newtheorem{definition}[theorem]{Definition}
\newtheorem{conjecture}[theorem]{Conjecture}
\theoremstyle{remark}
\newtheorem{remark}[theorem]{Remark}
\newtheorem*{main_remark}{Remark}
\newtheorem{corollary}[theorem]{Corollary}
\newtheorem*{main_corollary}{Corollary \ref{cor:51}}

\renewcommand{\thepart}{\Roman{part}}
\setcounter{tocdepth}{1}
\setcounter{section}{0} 


\title[Derivatives of Eisenstein series of genus $g$ and weight $\frac{g}{2}$]{Derivatives of Eisenstein series of weight $2$ and intersections of modular correspondences}

\author[Sungmun Cho, Shunsuke Yamana and Takuya Yamauchi]{Sungmun Cho, Shunsuke Yamana and Takuya Yamauchi}

\address{Sungmun Cho, Department of
Mathematics, Kyoto University, JAPAN}
\email{sungmuncho12@gmail.com}

\address{Shunsuke Yamana, Hakubi Center, Yoshida-Ushinomiya-cho, Sakyo-ku, Kyoto, 606-8501, Japan}
\address{Department of mathematics, Kyoto University, Kitashirakawa Oiwake-cho, Sakyo-ku, Kyoto, 606-8502, Japan}
\address{Max Planck Institut f\"{u}r Mathematik, Vivatsgasse 7, 53111 Bonn, Germany}
\email{yamana07@math.kyoto-u.ac.jp}
\address{Takuya Yamauchi,
Mathematical Institute, Tohoku University, 
 6-3,Aoba, Aramaki, Aoba-Ku, Sendai 980-8578, JAPAN}
\email{yamauchi@math.tohoku.ac.jp}
\begin{abstract}
We give a formula for certain values and derivatives of Siegel series and use them to compute Fourier coefficients of derivatives of the Siegel Eisenstein series of weight $\frac{g}{2}$ and genus $g$. 
When $g=4$, the Fourier coefficient is approximated by a certain Fourier coefficient of the central derivative of the Siegel Eisenstein series of weight $2$ and genus $3$, which is related to the intersection of $3$ arithmetic modular correspondences.  
Applications include a relation between weighted averages of representation numbers of symmetric matrices. 
\end{abstract}
\maketitle


\section{Introduction}
\subsection{Motivation : On the modular correspondences}
Let $j=j'=j(\tau)$ be the elliptic modular function on the upper half plane. 
For $m\geq 1$ let $\vph_m\in\ZZ[j,j']$ be the classical modular polynomial defined by 
\[\vph_m(j(\tau),j(\tau'))=\prod_{A\in\Mat_2(\ZZ)\pmod{\SL_2(\ZZ)},\;\det A=m}(j(\tau)-j(A\tau')). \]
Put $S=\Spec\ZZ[j,j']$ and $S_\CC=\Spec\CC[j,j']$. 
Let $T_m$ and $T_{m,\CC}$ be the arithmetic and geometric divisors defined by $\vph_m=0$. 
We can view $S$ as an arithmetic threefold $\cals=\calm\times_{\Spec\ZZ}\calm$, where $\calm$ is the moduli stack of elliptic curves over $\ZZ$, and $T_m$ as the moduli stack $\calt_m$ of isogenies of elliptic curves of degree $m$.  
In the 19th century Hurwitz has computed the intersection  
\[(T_{m_1,\CC}\cdot T_{m_2,\CC}):=\dim_\CC\CC[j,j']/(\vph_{m_1},\vph_{m_2}) \]
of complex curves. 
Gross and Keating \cite{GK} discovered that $(T_{m_1,\CC}\cdot T_{m_2,\CC})$ is related to the Fourier coefficients of the Siegel Eisenstein series of weight $2$ for $Sp_2(\ZZ)$. 
Moreover, they gave an explicit expression for the intersection 
\[(T_{m_1}\cdot T_{m_2}\cdot T_{m_3}):=\log\sharp\ZZ[j,j']/(\vph_{m_1},\vph_{m_2},\vph_{m_3}) \] 
of $3$ arithmetic modular correspondences. 
It is already mentioned in the introduction of \cite{GK} that computations of Kudla or Zagier strongly suggest that $\deg\scrz(B)$ equals  the $B$-th Fourier coefficient of the derivative of the Siegel Eisenstein series of weight $2$ for $Sp_3(\ZZ)$, up to  multiplication by a constant which is independent of $B$. 
A complete proof of this identity has been given in \cite{RW} (cf. \cite{KR}). 

The purpose of this paper is to compute the Fourier coefficients of the derivative of the Siegel Eisenstein series of weight $2$ for $Sp_4(\ZZ)$. 
One may expect that these coefficients are related to the intersection of $4$ modular correspondences. 
However, the number
\[\log\sharp\ZZ[j,j']/(\vph_{m_1},\vph_{m_2},\vph_{m_3},\vph_{m_4}), \]
does not seem to be naturally expanded to a sum over positive semi-definite symmetric half-integral matrices of size $4$ and does not seem to be a right object. 
The fiber product $\calt_{m_1}\times_\cals\calt_{m_2}\times_\cals\calt_{m_3}\times_\cals\calt_{m_4}$ has a disjoint sum decomposition according to the values of the fundamental matrices: 
\[\calt_{m_1}\times_\cals\calt_{m_2}\times_\cals\calt_{m_3}\times_\cals\calt_{m_4}=\bigsqcup_T\scrz(T), \] 
where $T$ extends over the set of positive semi-definite symmetric half-integral matrices of size $4$ with diagonal entries $m_1,m_2,m_3,m_4$. 
If $T$ is positive definite, then $\scrz(T)$ is empty unless $\det T$ is a square and $T$ is split except over a single prime. 
If $T$ is positive definite and $\det T$ is a square, then the $T$-th Fourier coefficient is zero unless $T$ is anisotropic only at a prime $p$, in which case the $T$-th Fourier coefficient is approximately equal to $\deg\scrz(T')$, where $T'$ is some positive semi-definite symmetric half-integral matrix of size $3$ (see Theorem \ref{thm:12}).  
Our result may imply that for each point of the intersection, where $4$ surfaces intersect properly, in a small neighborhood of the point, the intersection multiplicity behaves like the intersection multiplicity of $3$ surfaces of them.


\medskip

In the intervening years Kudla and others have gone a long way towards proving such relations in much greater generality. 
In \cite{K1}, he introduced a certain family of Eisenstein series of genus $g$ and weight $\frac{g+1}{2}$. 
They have an odd functional equation and hence have a natural zero at their center of symmetry. 
The central derivatives of such series, which he refers to as incoherent Eisenstein series, have a connection with arithmetic algebraic geometry of cycles on integral models of Shimura varieties attached to orthogonal groups of signature $(2,g-1)$, at least when $g\leq 4$. 
We refer the reader to \cite{KRY} for $g=1$, to \cite{K1,KR2,KRY2} for $g=2$, to \cite{KR,T1,RW} for $g=3$, and to \cite{KR3} for $g=4$. 
However, 
there are serious problems with the construction of arithmetic models of these Shimura varieties as soon as $g\geq 5$. 

\subsection{The Fourier coefficients of  derivative of Eisenstein series}
In this paper we compute the Fourier coefficients of derivatives of incoherent Eisenstein series of genus $g$ and weight $\frac{g}{2}$. 
In this introductory section we will consider classical Eisenstein series of level $1$. 
Let $g$ be a positive integer that is divisible by $4$. 
Let 
\[E_g(Z,s)=\sum_{\{C,D\}}\det(CZ+D)^{-g/2}|\det(CZ+D)|^{-s}(\det Y)^{s/2}\]
be the Siegel Eisenstein series of genus $g$, where $\{C,D\}$ runs over a complete set of representatives of the equivalence classes of coprime
symmetric pairs of degree $g$, and $Z$ is a complex symmetric matrix of degree $g$ with positive definite imaginary part $Y$. 
This series converges absolutely for $\Re s>\frac{g}{2}+1$ and admits a meromorphic continuation to the whole $s$-plane by the general theory of Langlands. 

If $\frac{g}{4}$ is even, then $E_g(Z,s)$ is holomorphic at $s=0$ and the $T$-th Fourier coefficient of $E_g(Z,0)$ is equal to 
\beq
2\biggl(\sum_i\frac{1}{N(L_i,L_i)}\biggl)^{-1}\sum_i\frac{N(L_i,T)}{N(L_i,L_i)} \label{tag:11}
\eeq
by the Siegel formula (see \cite{Si1,KuR,Y1}), where $\{L_i\}$ is the set of isometry classes of positive definite even unimodular lattices of rank $g$. 
Here $N(L,L')$ denotes the number of isometries $L'\to L$ for two quadratic spaces $L,L'$ over $\ZZ$. 
In particular, the nondegenerate Fourier coefficients are supported on a single rational equivalence class.  

On the other hand, if $\frac{g}{4}$ is odd, then $E_g(Z,s)$ has a zero at $s=0$. 
Our main object of study in this paper is the derivative  
\[\frac{\partial}{\partial s}E_g(Z,s)|_{s=0}=\sum_{T>0}C_g(T)e^{2\pi\iu\tr(TZ)}+\sum_{\text{other }T}C_g(T,Y)e^{2\pi\iu\tr(TZ)}. \]

Fix a positive definite symmetric half-integral $n\times n$ matrix $T$ and a rational prime $p$. 
Let $\QQ^{(p)}$ be a subring of $\QQ$, consisting of the numbers of the form $\frac{a}{p^{n}}$ with $n\in\NN$ and $a\in\ZZ$.
We define the additive character $\bfe_p$ of $\QQ_p$ by setting $\bfe_p(x)=e^{-2\pi\iu y}$ with $y\in\QQ^{(p)}$ such that $x-y\in\ZZ_p$. 
The Siegel series attached to $T$ and $p$ is defined by
\[b_p(T,s)=\sum_{z\in\Sym_n(\QQ_p)/\Sym_n(\ZZ_p)} \bfe_p(-\tr(Tz))\nu[z]^{-s},\]
where $\nu[z]$ is the product of denominators of elementary divisors of $z$. 
Put $D_T=(-4)^{[n/2]}\det T$. 
We denote the primitive Dirichlet character corresponding to $\QQ(\sqrt{D_T})$ by $\chi_T$ and its conductor by $\frkd^T$. 
Put $\xi^T_p=\chi_T(p)$. 
Let $e^T_p=\ord_pD_T$ or $e^T_p=\ord_pD_T-\ord_p\frkd^T$ according as $n$ is odd or even. 
There exists a polynomial $F_p^T(X)\in\ZZ[X]$ such that 
\[b_p(T,s)=\gam^T_p(p^{-s})F^T_p(p^{-s}), \]
where
\[\gam^T_p(X)
=(1-X)\prod_{j=1}^{[n/2]}(1-p^{2j}X^2)\times
\begin{cases}
1 &\text{if $n$ is odd, } \\
\frac{1}{1-\xi^T_p p^{n/2}X} &\text{if $n$ is even. }
\end{cases} \]
The symbol $\eta_p^T$ stands for the normalized Hasse invariant of $T$ over $\QQ_p$ (see Definition \ref{def:21}). 
We write $\Dif(T)$ for the finite set of prime numbers $p$ such that $\eta_p^T=-1$. 
A direct calculation gives the following formula:  

\begin{main_proposition}
Assume that $\frac{g}{4}$ is odd. 
Let $T$ be a positive definite symmetric half-integral matrix of size $g$. 
\begin{enumerate}
\renewcommand\labelenumi{(\theenumi)}
\item If $\chi_T=1$, then $C_g(T)=0$ unless $\Dif(T)$ is a singleton. 
\item If $\chi_T=1$ and $\Dif(T)=\{p\}$, then 
\[C_g(T)=-\frac{2^{(g+2)/2}p^{-(g+e^T_p)/2}\log p}{\zet\bigl(1-\frac{g}{2}\bigl)\prod_{i=1}^{(g-2)/2}\zet(1-2i)}\frac{\partial F^T_p}{\partial X}(p^{-g/2})\prod_{p\neq \ell|D_T}\ell^{-e^T_\ell/2}F_\ell^T(\ell^{-g/2}). \]
\item If $\chi_T\neq 1$, then  
\[C_g(T)=-\frac{2^{(g+2)/2}L(1,\chi_T)}{\zet\bigl(1-\frac{g}{2}\bigl)\prod_{i=1}^{(g-2)/2}\zet(1-2i)}\prod_{\ell|D_T}p^{-e^T_\ell/2}F_\ell^T(\ell^{-g/2}). \]
\end{enumerate}
\end{main_proposition}

\begin{remark}\label{rem:11}
If $\chi_T\neq 1$, then $L(1,\chi_T)=\frac{\sqrt{\frkd^T}}{\log\eps}h$ by Dirichlet's class number formula, where $h$ is the class number of the real quadratic field $\QQ(\sqrt{\det T})$ and $\eps=\frac{t+u\sqrt{\frkd^T}}{2}$ ($t>0$, $u>0$) is the solution to the Pell equation $t^2-\frkd^Tu^{2}=4$ for which $u$ is smallest. 
\end{remark}

The following theorem is a special case of Theorem \ref{thm:41} and allows us to compute $\frac{\partial F^T_p}{\partial X}(\xi_p^Tp^{-g/2})$. 
For simplicity we here assume $p$ to be odd. 

\begin{theorem}\label{thm:11}
Let $p$ be an odd rational prime and $T=\diag[t_1,\dots,t_g]$ with $0\leq\ord_pt_1\leq\cdots\leq\ord_pt_g$. 
Put $T'=\diag[t_1,\dots,t_{g-1}]$.  
Suppose that $g$ is even and $p\nmid\frkd^T$. 
Then
\[F_p^T(\xi_p^Tp^{-g/2})=p^{e^T_p/2}F_p^{T'}(\xi_p^Tp^{-g/2}). \]
If $\eta^T_p=-1$, then
\[\frac{\xi^T_p}{p^{g/2}}\frac{\partial F_p^T}{\partial X}\biggl(\frac{\xi^T_p}{p^{g/2}}\biggl)
=\frac{F_p^{T'}(\xi^T_p p^{(2-g)/2})}{p-1}
-p^{e_p^T/2}\frac{\xi^T_p}{p^{g/2}}\frac{\partial F_p^{T'}}{\partial X}\biggl(\frac{\xi^T_p}{p^{g/2}}\biggl). \]
\end{theorem}

Our key ingredient is the explicit formula for $F_p^T(X)$, given by Ikeda and Katsurada in \cite{IK2}, which expresses the polynomial $F_p^T$ in terms of the (naive) extended Gross--Keating datum $H$ of $T$ over $\ZZ_p$. 
The polynomial $F_p^{T'}=F_p^{H'}$ is defined in terms of a subset $H'\subsetneq H$ for any $p$ in a uniform way. 
Actually, if $g=4$, then the values $\frac{\partial F^{H'}_p}{\partial X}(p^{-2})$ and $F_p^{H'}(p^{-1})$ depend only on $(a_1,a_2,a_3)$ if we write $(a_1,a_2,a_3,a_4)$ for the Gross--Keating invariant of $T$ over $\ZZ_p$.

\subsection{Applications}
\subsubsection{On the average of the representation numbers}
Theorem \ref{thm:11} combined with the Siegel formula will identify (\ref{tag:11}) with four times the average of the representation numbers of a symmetric matrix of size $g-1$ (see Conjecture \ref{conj:51} and Proposition \ref{prop:52}). 
The following result is a special case of Proposition \ref{prop:52}. 

\begin{main_corollary}
If $T$ is a positive definite symmetric half-integral matrix of size $4$ which satisfies $\chi^T=1$ and $\eta^T_\ell=1$ for $\ell\neq p$, then there exists a positive definite symmetric half-integral matrix $T'$ of size $3$ such that 
\[\sum_{(E',E)}\frac{N(\Hom(E',E),T)}{\sharp\Aut(E)\sharp\Aut(E')}=2\sum_{(E',E)}\frac{N(\Hom(E',E),T')}{\sharp\Aut(E)\sharp\Aut(E')}, \]
where $(E,E')$ extends over all pairs of isomorphism classes of supersingular elliptic curves over $\bar\FF_p$. 
\end{main_corollary}

\subsubsection{On the Fourier coefficients and the modular correspondences}
The factor $\frac{\partial F_p^{H'}}{\partial X}(\xi^T_pp^{-g/2})$ appears in Fourier coefficients of central derivatives of incoherent Eisenstein series of genus $g-1$ and weight $\frac{g}{2}$, which have close connection with arithmetical geometry on Shimura varieties at least for $g\leq 5$ as mentioned above. 
We will be mostly interested in the case $g=4$. 
When $T_{m_1}$, $T_{m_2}$ and $T_{m_3}$ intersect properly, the formula of Gross and Keating in \cite{GK} can be stated as follows:  
\[(T_{m_1}\cdot T_{m_2}\cdot T_{m_3})=\sum_B\deg\scrz(B), \]
where $B$ extends over all positive definite symmetric half-integral matrices with diagonal entries $m_1,m_2,m_3$. 
Here $\deg\scrz(B)=0$ unless $\Dif(B)$ consists of a single rational prime $p$, in which case 
\beq
\deg\scrz(B)=-\frac{(\log p)}{2p^2}\frac{\partial F^B_p}{\partial X}\left(\frac{1}{p^2}\right)\sum_{(E,E')}\frac{N(\Hom(E',E),B)}{\sharp\Aut(E)\sharp\Aut(E')}. \label{tag:12}
\eeq
The degree $\deg\scrz(B)$ equals the $B$-th Fourier coefficient of the derivative of the Siegel Eisenstein series of weight $2$ and genus $3$ up to a negative constant (cf. Theorem 2.2 of \cite{RW}). 
We combine (\ref{tag:12}), Theorem \ref{thm:51} and Corollary \ref{cor:51} to obtain the following formula: 

\begin{theorem}\label{thm:12}
If $T$ is a positive definite symmetric half-integral matrix of size $4$, $\chi_T=1$ and $\Dif(T)$ consists of a single prime number $p$, then there exists a positive definite symmetric half-integral matrix $T'$ of size $3$ such that 
\[\frac{C_4(T)}{-2^8\cdot 3^2}=\deg\scrz(T')+\frac{F_p^{T'}(p^{-1})}{2\sqrt{p}^{e^T_p}(p-1)}\log p\sum_{(E,E')}\frac{N(\Hom(E',E),T')}{\sharp\Aut(E)\sharp\Aut(E')}, \]
where $(E,E')$ extends over all pairs of isomorphism classes of supersingular elliptic curves over $\bar\FF_p$. 
\end{theorem}

Since $\Hom(E',E)$ is a quaternary quadratic space, if $S$ has rank greater than $4$, then $N(\Hom(E,E'),S)=0$. 
Therefore when $g\geq 5$, the nature of Fourier coefficients of the derivative of Eisenstein series of weight $2$ and genus $g$ should be much different. 
The case $g=4$ should be a boundary case. 
We will explicitly compute $F_p^{T'}(p^{-1})$ in Lemma \ref{lem:53} and show that 
\[\biggl|\frac{C_4(T)}{-2^8\cdot 3^2\cdot \deg\scrz(T')}-1\biggl|<\frac{20}{p\sqrt{p}}. \] 
Moreover, Corollary \ref{cor:52} says that for a fixed prime number $p$
\[\lim_{\ord_p(\det T)\to\infty}\frac{C_4(T)}{-2^8\cdot 3^2\cdot\deg\scrz(T')}=1. \] 

\subsection{Organizations}
We now explain the lay-out of this paper. 
Section \ref{sec:2} extends the notion of incoherent Eisenstein series to the case where the point at which the Eisenstein series is evaluated lies within the left half-plane. 
We calculate the Fourier coefficients of those Eisenstein series and their derivatives. 
In Section \ref{sec:3} we derive a general formula for Fourier coefficients of derivatives of incoherent Eisenstein series. 
Section \ref{sec:4} is devoted to a local study of the Siegel series.
We give the inductive expression for the special value of the derivative of the Siegel series. 
Section \ref{sec:5} is devoted to proving Theorem \ref{thm:51}. 

\subsection*{Acknowledgement}
Cho is  supported by JSPS KAKENHI Grant No. 16F16316.
Yamana is partially supported by JSPS Grant-in-Aid for Young Scientists (B) 26800017. 
We would like to thank Stephen Kudla for very stimulating discussions. 


\section*{\bf Notations}

For a finite set $A$, we denote by $\sharp A$ the number of elements in $A$. 
For a ring $R$ we denote by $\Mat_{i,j}(R)$ the set of $i\times j$-matrices with entries in $R$ and write $\Mat_m(R)$ in place of $\Mat_{m,m}(R)$. 
The group of all invertible elements of $\Mat_m(R)$ and the set of symmetric matrices of size $m$ with entries in $R$ are denoted by $\GL_m(R)$ and $\Sym_m(R)$, respectively. 
Let $\cale_m(R)$ be the set of elements $(a_{ij})\in\Sym_m(R)$ such that $a_{ii}\in 2R$ for every $i$. 
For matrices $B\in\Sym_m(R)$ and $G\in\Mat_{m,n}(R)$ we use the abbreviation $B[G]=\trs GBG$, where $\trs G$ is the transpose of $G$.  
If $A_1, \dots, A_r$ are square matrices, then $\diag[A_1, \dots, A_r]$ denotes the matrix with $A_1, \dots, A_r$ in the diagonal blocks and $0$ in all other blocks.
Let $\ono_m$ be the identity matrix of degree $m$. 
Put 
\begin{align*}
Sp_g(R)&=\left\{G\in\GL_{2g}(R)\;\biggl|\;G\begin{pmatrix} 0 & \ono_g \\ -\ono_g & 0\end{pmatrix}\trs G=\begin{pmatrix} 0 & \ono_g \\ -\ono_g & 0\end{pmatrix}\right\}, \\
M_g(R)&=\left\{\bfm(A)=\begin{pmatrix} A & 0 \\ 0 & \trs A^{-1}\end{pmatrix}\;\biggl|\;A\in\GL_g(R)\right\}, \\
N_g(R)&=\left\{\bfn(B)=\begin{pmatrix} \ono_g & B \\ 0 & \ono_g\end{pmatrix}\;\biggl|\;B\in\Sym_g(R)\right\}. 
\end{align*} 

Let $\ZZ$ be the set of integers and $\mu_n$ the group of $n$-th roots of unity. 
If $x$ is a real number, then we put $[x]=\max\{m\in\ZZ\;|\;m\leq x\}$. 



\section{Eisenstein series}\label{sec:2}

Let $k$ be a totally real number field with integer ring $\frko$. 
The set of real places of $k$ is denoted by $\frkS_\infty$. 
The completion of $k$ at a place $v$ is denoted by $k_v$. 
Let $(\;,\;)_{k_v}:k_v^\times\times k_v^\times\to\mu_2$ denote the Hilbert symbol.
We let $\frkp$ denote a finite prime of $k$ and do not use the letter $\frkp$ for a real place. 
Let $q_\frkp=\sharp\frko/\frkp$ be the order of the residue field.  
We define the character $\bfe_\frkp$ of $k_\frkp$ by $\bfe_\frkp(x)=\bfe(-y)$ with $y\in\QQ^{(p)}$ such that $\Tr_{k_\frkp/\QQ_p}(x)-y\in\ZZ_p$ if $p$ is the rational prime divisible by $\frkp$. 
Put $\bfe(z)=e^{2\pi\iu z}$ for $z\in\CC$ and $\bfe_\infty(z)=\prod_{v\in\frkS_\infty}\bfe(z_v)$ for $z\in\prod_{v\in\frkS_\infty}\CC$.  

Once and for all we fix a positive integer $g\geq 2$. 
Let $(V,(\;,\;))$ be a quadratic space of dimension $m$ over $k_v$. 
Whenever we speak of a quadratic space, we always assume that $(\;,\;)$ is nondegenerate, i.e., $(u,V)=0$ implies that $u=0$.  
Put $s_0=\frac{1}{2}(m-g-1)$. 
Given $u=(u_1,\dots,u_g)\in V^g$, we write $(u,u)$ for the $g\times g$ symmetric matrix with $(i,j)$ entry equal to $(u_i,u_j)$. 
We write $\det V$ for the element in $k_v^\times/k_v^{\times2}$ represented by the determinant of the matrix representation of the bilinear form $(\;,\;)$ with respect to any basis for $V$ over $k_v$. 
We define the character $\chi^V:k_v^\times\to\mu_2$ by 
\beq
\chi^V(t)=(t,(-1)^{m(m-1)/2}\det V)_{k_v}. \label{tag:20}
\eeq  

We normalize our Hasse invariant $\eta^V$ so that it depends only on the isomorphism class of an anisotropic kernel of $V$ (cf. \cite{E1,Sh6}). 

\begin{definition}\label{def:21}
We associate to the quadratic space $V$ over $k_\frkp$ of dimension $m$ an invariant $\eta^V\in\mu_2$ according to the type of $V$ as follows:  
\begin{itemize}
\item If $m$ is odd, then an anisotropic kernel of $V$ has dimension $2-\eta^V$. 
\item If $m$ is even and $\chi^V\neq 1$ and if we choose an element $c\in k_\frkp^\times$ such that $\chi^V(c)=\eta^V$, then $V$ is the orthogonal sum of a split form of dimension $m-2$ with the norm form scaled by the factor $c$ on the quadratic extension of $k_\frkp$ corresponding to $\chi^V$. 
\item If $m$ is even and $\chi^V=1$, then $V$ is split or the orthogonal sum of the norm form on the quaternion algebra over $k_\frkp$ with a split form of dimension $m-4$ according as $\eta^V=1$ or $-1$. 
\end{itemize}
\end{definition}

We denote the set of positive definite symmetric matrices over $\RR$ of rank $g$ by $\Sym_g(\RR)^+$. 
Let 
\[\frkH_g=\{X+\iu Y\in\Sym_g(\CC)\;|\;Y\in\Sym_g(\RR)^+\}\] 
be the Siegel upper half-space of genus $g$. 
The real symplectic group $Sp_g(\RR)$ acts transitively on $\frkH_g$ by $GZ=(AZ+B)(CZ+D)^{-1}$ for $Z\in\frkH_g$ and $G=\begin{pmatrix} A & B \\ C & D \end{pmatrix}\in Sp_g(\RR)$. 
We define the maximal compact subgroups by   
\begin{align*}
K_\frkp&=Sp_g(\frko_\frkp), & 
K_v&=\{G\in Sp_g(k_v)\;|\;G(\iu\ono_g)=\iu\ono_g\} 
\end{align*} 
for $v\in\frkS_\infty$. 
We have the Iwasawa decomposition 
\[Sp_g(k_v)=M_g(k_v)N_g(k_v)K_v. \]
Denote the two-fold metaplectic cover of $Sp_g(k_v)$ by $\Mp_v$. 
There is a canonical splitting $N_g(k_v)\to\Mp_v$. 
When $\frkp$ does not divide $2$, we have a canonical splitting $K_\frkp\to\Mp_\frkp$. 
We still use $N_g(k_v)$ and $K_\frkp$ to denote the images of these splittings. 
Let $\til K_v$ denote the pull-back of $K_v$ in $\Mp_v$. 
Define the map $\Mp_v\to\RR^\times_+$ by writing $\til G=\bfn(b)\til m\til k\in\Mp_v$ with $b\in\Sym_g(k_v)$, $a\in\GL_g(k_v)$, $\til m=(\bfm(a),\zet)$ and $\til k\in\til K_v$ and setting $|a(\til G)|=|\det a|_v$. 
We refer to Section 1.1 of \cite{Y1} for additional explanation. 

Let $V$ be a quadratic space over $k_v$ and $\ome_v$ the Weil representation of $\Mp_v$ with respect to $\bfe_v$ on the space $\cals(V^g)$ of the Schwartz functions on $V^g$. 
We associate to $\vph\in\cals(V^g)$ the function on $\Mp_v\times\CC$ by 
\[f_\vph^{(s)}(\til G)=(\ome_v(\til G)\vph)(0)|a(\til G)|^{s-s_0}. \]

The real metaplectic group acts on the half-space $\frkH_g$ through $Sp_g(\RR)$. 
There is a unique factor of automorphy $\jmath_v:\Mp_v\times\frkH_g\to\CC^\times$ whose square descends to the automorphy factor on $Sp(k_v)\times\frkH_g$ given by $\jmath_v(G_v,Z_v)^2=\det(C_vZ_v+D_v)$ for $G_v=\begin{pmatrix} * & * \\ C_v & D_v \end{pmatrix}\in Sp(k_v)$. 
We define an automorphy factor $\jmath:\prod_{v\in\frkS_\infty}(\Mp_v\times\frkH_g)\to\CC^\times$ by $\jmath(\til G,Z)=\prod_v\jmath_v(\til G_v,Z_v)$. 

Let $\AA$ be the adele ring of $k$ and $\AAf$ the finite part of the adele ring. 
We arbitrarily fix a quadratic character $\chi$ of $\AA^\times/k^\times$ such that $\chi_v=\sgn^{m(m-1)/2}$. 

\begin{definition}\label{def:22}
Let $\calc=\{\calc_v\}$ be a collection of local quadratic spaces of dimension $m$ such that $\chi^{\calc_v}=\chi_v$ for all $v$, such that $\calc_v$ is positive definite for $v\in\frkS_\infty$ and such that $\eta^{\calc_\frkp}=1$ for almost all $\frkp$. 
We say that $\calc$ is coherent if it is the set of localizations of a global quadratic space. 
Otherwise we call $\calc$ incoherent. 
\end{definition}

One can derive the following criterion from the theorem of Minkowski-Hasse (see Theorem 4.4 of \cite{Sh5}). 

\begin{lemma}\label{lem:21}
Put $d=[k:\QQ]$. When $m$ is odd, $\calc$ is coherent if and only if $(-1)^{d(m^2-1)/8}\prod_\frkp\eta^{\calc_\frkp}=1$. 
When $m$ is even, $\calc$ is coherent if and only if $(-1)^{dm(m-2)/8}\prod_\frkp\eta^{\calc_\frkp}=1$. 
\end{lemma}

There is a unique splitting $Sp_g(k)\hookrightarrow\Mp_g$ by which we regard $Sp_g(k)$ as the subgroup of the two-fold metaplectic cover $\Mp_g$ of $Sp_g(\AA)$. 
Let $P_g=M_gN_g$ be the Siegel parabolic subgroup of $Sp_g$. 
Given any pure tensor $\vph=\otimes_\frkp\vph_\frkp\in\otimes_\frkp'\cals(\calc_\frkp^g)$, we consider the function   
\begin{align*}
f_\vph^{(s)}(\til G)&=\prod_\frkp f_{\vph_\frkp}^{(s)}(\til G_\frkp), & 
f_{\vph_\frkp}^{(s)}(\til G_\frkp)&=(\ome_\frkp(\til G_\frkp)\vph_\frkp)(0)|a(\til G_\frkp)|^{s-s_0} 
\end{align*}
on $\Mp_g\times\CC$ and the Eisenstein series on $\prod_{v\in\frkS_\infty}\frkH_g$
\[E(Z,f_\vph^{(s)})=(\det Y)^{(s-s_0)/2}\sum_{\gam\in P_g(k)\bsl Sp_g(k)}|\jmath(\gam,Z)|^{s_0-s}\jmath(\gam,Z)^{-g}f_\vph^{(s)}(\gam), \]
where $Y$ is the imaginary part of $Z$. 
The series is absolutely convergent for $\Re s>\frac{g+1}{2}$. 
It admits a meromorphic continuation to the whole plane and its Laurent coefficients define automorphic forms. 
Moreover, it is holomorphic at $s=s_0$, and if $\calc$ is coherent, then the Siegel--Weil formula holds by \cite{KuR}. 

From now on we require that $m\leq g+1$. 
Let $V$ be a totally positive definite quadratic space of dimension $m$ over $k$. 
We normalize the invariant measure $\d h$ on $\O(V,k)\bsl\O(V,\AA)$ to have total volume $1$ and define the integral
\[I(Z,\vph)=\int_{\O(V,k)\bsl\O(V,\AA)}\Tht(Z,h;\vph)\,\d h \]
of the theta function 
\[\Tht(Z,h;\vph)=\sum_{u\in V(k)^g}\vph(h^{-1}u)\bfe_\infty(\tr((u,u)Z)). \]
Since we are under coherent situation, the Siegel--Weil formula can now be stated as follows: 
\beq
E(Z,f_\vph^{(s)})|_{s=s_0}=2I(Z,\vph). \label{tag:21}
\eeq
The reader who is interested in this identity can consult Theorem 2.2(\roman{one}) of \cite{Y1}. 
On the other hand, if $\calc$ is incoherent, then the series $E(Z,f_\vph^{(s)})$ has a zero at $s=s_0$ by Corollary 5.5 of \cite{Y1}. 

Consider the Fourier expansions
\begin{align*}
E(Z,f_\vph^{(s)})&=\sum_{T\in\Sym_g(k)}A(T,Y,\vph,s)\bfe_\infty(\tr(TZ)), \\
\frac{\partial}{\partial s}E(Z,f_\vph^{(s)})|_{s=s_0}&=\sum_{T\in\Sym_g(k)}C(T,Y,\vph)\bfe_\infty(\tr(TZ)), 
\end{align*}
where 
\begin{align*}
Z&=X+\iu Y, & 
C(T,Y,\vph)&=\frac{\partial}{\partial s}A(T,Y,\vph,s)|_{s=s_0}. 
\end{align*}
Put $\Sym_g^\nd=\Sym_g(k)\cap\GL_g(k)$.  
When $T\in\Sym_g^\nd$, the Fourier coefficient has an explicit expression as an infinite product
\[A(T,Y,\vph,s)=a(T,Y,s)\prod_\frkp W_T\Big(f_{\vph_\frkp}^{(s)}\Big) \]
for $\Re s\gg 0$, where 
\[W_T\Big(f_{\vph_\frkp}^{(s)}\Big)=\int_{\Sym_g(k_\frkp)}f_{\vph_\frkp}^{(s)}\left(\begin{pmatrix} 0 & \ono_g \\ -\ono_g & 0\end{pmatrix}\bfn(z_\frkp)\right)\overline{\bfe_\frkp(\tr(Tz_\frkp))}\,\d z_\frkp \]
and $a(T,Y,s)\bfe_\infty(\iu\tr(TY))$ is a product of the confluent hypergeometric functions investigated in \cite{Sh1}. 
Given $T\in\Sym_g^\nd$, we define the quadratic form on $V^T=k^g$ by $u\mapsto T[u]$ and define the Hecke character $\chi^T=\prod_v\chi^T_v$ and the Hasse invariants $\eta^T_\frkp$, where $\chi^T_v$ is defined in (\ref{tag:20}). 
Let $\Dif(T,\calc)$ denote the set of places $v$ of $k$ such that $T$ is not represented by $\calc_v$. 
Let $\Sym_g^+$ denote the set of totally positive definite symmetric $g\times g$ matrices over $k$. 

\begin{lemma}\label{lem:22}
Let $\vph_\frkp\in\cals(\calc_\frkp^g)$ and $T\in\Sym_g^\nd$. 
\begin{enumerate}
\renewcommand\labelenumi{(\theenumi)}
\item $a(T,Y,s)$ and $W_T\Big(f_{\vph_\frkp}^{(s)}\Big)$ are entire functions in $s$. 
\item\label{lem:221} $\lim_{s\to s_0}W_T\Big(f_{\vph_\frkp}^{(s)}\Big)=0$ unless $T$ is represented by $\calc_\frkp$. 
\item\label{lem:223} If $m=g$, $T\in\Sym_g^+$, $\chi^T=\chi$ and $\calc$ is incoherent, then $\Dif(T,\calc)$ is a finite set of odd cardinality. 
\end{enumerate}
\end{lemma}

\begin{proof}
The first part is well-known (see \cite{Kar1,Sh1}). 
Lemma on p. 73 of \cite{MVW} implies (\ref{lem:221}). 
By assumption $\Dif(T,\calc)=\{\frkp\;|\;\eta^{\calc_\frkp}=-\eta^T_\frkp\}$. 
Since $\calc$ is incoherent, Lemma \ref{lem:21} implies $\prod_\frkp\eta^{\calc_\frkp}=-\prod_\frkp\eta^T_\frkp$, which proves (\ref{lem:223}).  
\end{proof}

Let $T\in\Sym_g^+$. 
Then both $a(T,Y,s_0)$ and $C(T,Y,\vph)$ are independent of $Y$. 
Put 
\begin{align*}
c_m(T)&=a(T,Y,s_0), &  
C(T,\vph)&=C(T,Y,\vph), &
D_T&=\mathrm{N}_{k/\QQ}(\det(2T)). 
\end{align*}
Let $\frkd_k$ denote the absolute value of the discriminant of $k$. 
Note that  
\begin{align}
c_g(T)&=c_gD_T^{-1/2}, & 
c_g&=\frkd_k^{-g(g+1)/4}\biggl(\bfe\biggl(\frac{g^2}{8}\biggl)\frac{2^g\pi^{g^2/2}}{\Gam_g\bigl(\frac{g}{2}\bigl)}\biggl)^d \label{tag:22}
\end{align}
by (4.34K) of \cite{Sh1}, where $\Gam_g(s)=\pi^{g(g-1)/4}\prod_{i=0}^{g-1}\Gam\bigl(s-\frac{i}{2}\bigl)$.

\begin{proposition}\label{prop:21}
Let $m=g$ and $T\in\Sym_g^+$. 
Suppose that $\calc$ is incoherent. 
If $\chi^T=\chi$, then $C(T,\vph)=0$ unless $\Dif(T,\calc)$ is a singleton. 
Moreover, if $\Dif(T,\calc)=\{\frkp\}$, then 
\[C(T,\vph)=c_gD_T^{-1/2}\lim_{s\to -1/2}\frac{\partial W_T\Big(f_{\vph_\frkp}^{(s)}\Big)}{\partial s}\prod_{\frkl\neq\frkp}W_T\Big(f_{\vph_\frkl}^{(s)}\Big). \]
\end{proposition}

\begin{proof}
For given $\vph$ and $T$, let $\frkS$ be a finite set of rational primes of $k$ such that if $\frkq\notin\frkS$, then $\frkq$ does not divide $2$, $\chi_\frkq$ is unramified, $\bfe_\frkq$ is of order $0$, $T\in\GL_g(\frko_\frkq)$ and the restriction of $f^{(s)}_{\vph_\frkq}$ to $K_\frkq$ is $1$. 
Since $T$ cannot be unimodular at $\frkp\in\Dif(T,\calc)$, the set $\frkS$ necessarily contains $\Dif(T,\calc)$. 
The $T$-th Fourier coefficient of $E(Z,f_\vph^{(s)})$ is given by 
\beq 
A(T,Y,\vph,s)=\bet^T(s)a(T,Y,s)\prod_{\frkq\in \frkS}\bet^T_\frkq(s)W_T\Big(f_{\vph_\frkq}^{(s)}\Big), \label{tag:23}
\eeq
where 
\begin{align*}
\bet^T(s)&=\frac{L\bigl(s+\frac{1}{2},\chi^T\chi\bigl)}{\prod^{[(g+1)/2]}_{j=1}\zet(2s+2j-1)}\times
\begin{cases} 
1 &\text{if $2\nmid g$, }\\
L\bigl(s+\frac{g+1}{2},\chi\bigl)^{-1} &\text{if $2|g$, }
\end{cases}\\
\bet^T_\frkq(s)&=\frac{\prod^{[(g+1)/2]}_{j=1}\zet_\frkq(2s+2j-1)}{L\bigl(s+\frac{1}{2},\chi^T_\frkq\chi_\frkq\bigl)}\times
\begin{cases} 
1 &\text{if $2\nmid g$, }\\
L\bigl(s+\frac{g+1}{2},\chi_\frkq\bigl) &\text{if $2|g$. }
\end{cases}
\end{align*}

Notice that the product $\bet^T_\frkq(s)W_T\Big(f_{\vph_\frkq}^{(s)}\Big)$ is holomorphic at $s=-\frac{1}{2}$. 
Indeed, if $\chi^T_\frkq=\chi_\frkq$, then $\bet^T_\frkq(s)$ is holomorphic at $s=-\frac{1}{2}$ while if $\chi^T_\frkq\neq\chi_\frkq$, then $\bet^T_\frkq(s)$ has a simple pole at $s=-\frac{1}{2}$, but $W_T\Big(f_{\vph_\frkq}^{(s)}\Big)$ has a zero at $s=-\frac{1}{2}$ by Lemma \ref{lem:22}(\ref{lem:221}). 

Assume that $\chi^T=\chi$. 
Then $\bet^T(s)$ is holomorphic and has no zero at $s=-\frac{1}{2}$. 
If $\frkq\in\Dif(T,\calc)$, then $\bet^T_\frkq(s)W_T\Big(f_{\vph_\frkq}^{(s)}\Big)$ has a zero at $s=-\frac{1}{2}$ by Lemma \ref{lem:22}(\ref{lem:221}), which combined with (\ref{tag:23}) proves the first statement. 
We obtain the first formula by differentiating (\ref{tag:23}) at $s=-\frac{1}{2}$. 
\end{proof}

\begin{corollary}\label{cor:21}
If $m=g$, $\calc$ is incoherent and $T\in\Sym_g^+$ with $\chi^T\neq\chi$, then 
\[C(T,\vph)=c_gD_T^{-1/2}\lim_{s\to -1/2}\frac{\partial \bet^T}{\partial s}(s)\prod_\frkp \bet^T_\frkp(s)W_T\Big(f_{\vph_\frkp}^{(s)}\Big). \]
\end{corollary}

\begin{proof}
Since $\bet^T(s)$ has a zero at $s=-\frac{1}{2}$ if $\chi\neq\chi^T$, we can deduce Corollary \ref{cor:21} from (\ref{tag:23}). 
\end{proof}


\section{Fourier coefficients of derivatives of Eisenstein series}\label{sec:3}

Let $\gam_v(t)$ be the Weil constant associated to the character of second degree $u\mapsto\bfe_v(tu^2)$, and $\vep_v(\calc_v)$ the unnormalized Hasse invariant of $\calc_v$. 
Put  
\[\gam(\calc_v)=\vep_v(\calc_v)\gam_v\left(\frac{1}{2}\right)^{m-1}\gam_v\left(\frac{1}{2}\det\calc_v\right). \]
Let $L_\frkp$ be an integral lattice of $\calc_\frkp$, i.e., a finitely generated $\frko_\frkp$-submodule of $\calc_\frkp$ which spans $\calc_\frkp$ over $k_\frkp$ and such that $(u,u)\in\frko_\frkp$ for every $u\in L_\frkp$. 
Let 
\[L_\frkp^*=\{u\in\calc_\frkp\;|\;2(u,w)\in\frko_\frkp\text{ for every }w\in L_\frkp\}\]
be its dual lattice. 
Let $\ch\La L_\frkp^g\Ra\in\cals(\calc_\frkp^g)$ be the characteristic function of $L_\frkp^g$. 
We write $S_\frkp$ for the matrix for the quadratic form on $\calc_\frkp$ with respect to a fixed basis of $L_\frkp$. 
For nondegenerate symmetric matrices $T\in\frac{1}{2}\cale_g(\frko_\frkp)$ and $S\in\frac{1}{2}\cale_m(\frko_\frkp)$ the local density of representing $T$ by $S$ is defined by  
\[\alp_\frkp(S,T)=\lim_{i\to\infty}q_\frkp^{ig((g+1)-2m)/2}A_i(S,T), \]
where  
\[A_i(S,T)=\sharp\{X\in\Mat_{m,g}(\frko/\frkp^i)\;|\;S[X]\equiv T\pmod{\frkp^i}\}. \]

\begin{proposition}[cf. \cite{K1}]\label{prop:31}
Put $\calv_r=\calc_\frkp\oplus\calh(k_\frkp)^r$, where $\calh$ is the split binary quadratic space. 
We choose an integral lattice $L_\frkp^g\oplus\Mat_{2r,g}(\frko_\frkp)$ of full rank in $\calv_r^g$.
Then 
\[\lim_{s\to r+s_0}W_T\Big(f_{\ch\La L_\frkp^g\oplus\Mat_{2r,g}(\frko_\frkp)\Ra}^{(s)}\Big)=\frac{\alp_\frkp\left(S_\frkp\perp\frac{1}{2}\begin{pmatrix} & \ono_r \\ \ono_r & \end{pmatrix} ,T\right)}{\gam(\calc_\frkp)^g\frkd_k^{-g/2}[L_\frkp^*:L_\frkp]^{g/2}}. \]
Here, $s_0$ is associated to $\calc_\frkp$.
\end{proposition}

\begin{proof}
This result can be deduced from the proof of \cite[Lemma 8.3(2)]{Y3}. 
\end{proof}

Let $\calv$ be a totally positive definite quadratic space of dimension $g$ over $k$. 
Fix an integral lattice $L$ in $\calv$. 
Put 
\begin{align*}
L_\frkp&=L\otimes_\frko\frko_\frkp, & 
\ch\La L^g\Ra&=\otimes_\frkp\ch\La L^g_\frkp\Ra. 
\end{align*}
For $h\in\O(\calv,\AA)$ we write $hL$ for the lattice defined by $(hL)_\frkp=h_\frkp L_\frkp$. 
Put 
\begin{align*}
K_L&=\{h\in\SO(\calv,\AA)\;|\; hL=L\}, & 
\SO(L)&=\{h\in\SO(\calv,k)\;|\; hL=L\}.  
\end{align*}

\begin{definition}\label{def:31}
We mean by the genus (resp. class) of $L$ the set of all lattices of the form $hL$ with $h\in\O(\calv,\AA)$ (resp. $h\in\O(\calv,k)$). 
The proper class of $L$ consists of all lattices of the form $hL$ with $h\in\SO(\calv,k)$. 
\end{definition}

We write $\Xi'(L)$ and $\Xi(L)$ for the sets of classes and proper classes in the genus of $L$, respectively. 
Define the mass of the genus of $L$ by 
\begin{align*}
\frkm'(L)&=\sum_{\scrl\in\Xi'(L)}\frac{1}{\sharp\O(\scrl)}, & 
\frkm(L)&=\sum_{\scrl\in\Xi(L)}\frac{1}{\sharp\SO(\scrl)}. 
\end{align*}

\begin{remark}
For each finite prime $\frkp$ there is $h\in\O(\calv,k_\frkp)$ with $\det h=-1$ such that $hL_\frkp=L_\frkp$. 
The genus of $L$ therefore consists of lattices $hL$ with $h\in\SO(\calv,\AA)$. 
We identify $\Xi(L)$ with double cosets for $\SO(\calv,k)\bsl \SO(\calv,\AA)/K_L$ via the map $h\mapsto hL$. 
\end{remark}

Lemma 5.6(1) of \cite{Sh3} says that 
\beq
\frkm(L)=2\frkm'(L). \label{tag:31}
\eeq
We consider the following sums of representation numbers of $T\in\Sym_g(k)$: 
\begin{align*} 
R'(L,T)&=\sum_{\scrl\in\Xi'(L)}\frac{N(\scrl,T)}{\sharp\O(\scrl)}, & 
R(L,T)&=\sum_{\scrl\in\Xi(L)}\frac{N(\scrl,T)}{\sharp\SO(\scrl)},  
\end{align*}
where $N(L,T)=\sharp\{u\in L^g\;|\;(u,u)=T\}$. 

\begin{proposition}\label{prop:32}
Notation being as above, we have 
\[2\frac{R(L,T)}{\frkm(L)}=c_gD_T^{-1/2}\lim_{s\to -1/2}\prod_\frkp W_T\Big(f_{\ch\La L_\frkp^g\Ra}^{(s)}\Big). \]
\end{proposition}

\begin{proof}
This equality is nothing but the Siegel formula. 
Nevertheless we reproduce its proof here because of its importance for us. 
Since both sides are zero unless $V^T\simeq\calv$ by Lemma \ref{lem:22}(\ref{lem:221}), we may identify $V^T$ with $\calv$.
As is well-known, there exists $h\in\O(V^T,k_\frkp)$ such that $hL_\frkp=L_\frkp$ and $\det h=-1$. 
Since $\SO(V^T,\AA)\bsl\O(V^T,\AA)=\mu_2(\AA)$, we have 
\begin{align*}
I(Z,\ch\La L^g\Ra)&=\frac{1}{2}\int_{\SO(V^T,k)\bsl\SO(V^T,\AA)}\Tht(Z,h;\ch\La L^g\Ra)\,\d h.  
\end{align*}
Choose a finite set of double coset representatives $h_i\in \SO(V^T,\AAf)$ so that 
\[\SO(V^T,\AA)=\bigsqcup_i\SO(V^T,k)h_iK_L. \]
Then 
\begin{align*}
I(Z,\ch\La L^g\Ra)=\frac{1}{2}\vol(K_L)\sum_i\frac{\Tht(Z,h_i;\ch\La L^g\Ra)}{\sharp\SO(h_iL)}. 
\end{align*}
Since $\frkm(L)=2\vol(K_L)^{-1}$, the $T$-th Fourier coefficient of $I(Z,\ch\La L^g\Ra)$ is equal to $\frac{R(L,T)}{\frkm(L)}$. 
The Siegel--Weil formula (\ref{tag:21}) proves the declared identity. 
\end{proof}

An examination of the proof of Proposition \ref{prop:32} confirms that 
\beq
\frac{R(L,T)}{\frkm(L)}=\frac{R'(L,T)}{\frkm'(L)}. \label{tag:32}
\eeq

We can prove the following result by combining Propositions \ref{prop:21} and \ref{prop:32}. 

\begin{proposition}\label{prop:33}
We assume that $\Dif(T,\calc)=\{\frkp\}$, notation and assumption being as in Proposition \ref{prop:21}.  
Take an integral lattice $L$ in $V^T$ such that 
\[\lim_{s=-1/2}W_T\Big(f_{\ch\La L^g_\frkp\Ra}^{(s)}\Big)\neq 0. \]
If $\vph^{}_\frkl=\ch\La L_\frkl^g\Ra$ for every prime ideal $\frkl$ distinct from $\frkp$, then 
\[C(T,\vph)=2\frac{R(L,T)}{\frkm(L)}\lim_{s\to -1/2}W_T\Big(f_{\ch\La L^g_\frkp\Ra}^{(s)}\Big)^{-1}\frac{\partial W_T\Big(f_{\vph_\frkp}^{(s)}\Big)}{\partial s}. \]
\end{proposition}


\section{Siegel series}\label{sec:4}

In this section we drop the subscript $_\frkp$. 
Thus $k$ is a nonarchimedean local field of characteristic zero with integer ring $\frko$. 
We denote the maximal ideal of $\frko$ by $\frkp$ and the order of the residue field $\frko/\frkp$ by $q$. 
Fix a prime element $\vpi$ of $\frko$. 
We define the additive order $\ord:k^\times\to\ZZ$ by $\ord(\vpi^i\frko^\times)=i$. 

Let $T\in\frac{1}{2}\cale_g(\frko)$ with $\det T\neq 0$. 
Denote the conductor of $\chi^T$ by $\frkd^T$. 
Put
\begin{align*}
D_T&=(-4)^{[g/2]}\det T, \\
e^T&=\begin{cases}
\ord D_T &\text{if $g$ is odd, }\\ 
\ord D_T-\ord\frkd^T &\text{if $g$ is even, } 
\end{cases}\\
\xi^T&=\begin{cases}
1 &\text{if $D_T\in k^{\times 2}$, }\\
-1 &\text{if $D_T\notin k^{\times 2}$ and $\frkd^T=\frko$, }\\  
0 &\text{if $D_T\notin k^{\times 2}$ and $\frkd^T\neq\frko$. }
\end{cases}  
\end{align*}
 
The Siegel series associated to $T$ is defined by
\[b(T,s)=\sum_{z\in\Sym_g(k)/\Sym_g(\frko)} \psi(-\tr(T z))\nu[z]^{-s}, \]
where $\nu[z]=[z\frko^g+\frko^g:\frko^g]$ and $\psi$ is an arbitrarily fixed additive character on $k$ which is trivial on $\frko$ but nontrivial on $\frkp^{-1}$. 
As is well-known, there exists a polynomial $\bet(T,X)\in\ZZ[X]$ such that $\bet(T,q^{-s})=b(T,s)$. 
Moreover, this polynomial $\bet(T,X)$ is divisible by the following polynomial 
\[\gam^T(X)
=(1-X)\prod_{j=1}^{[g/2]}(1-q^{2j}X^2)\times
\begin{cases}
1 &\text{if $g$ is odd, } \\
\frac{1}{1-\xi^T q^{g/2}X} &\text{if $g$ is even. }
\end{cases} \]
Put
\begin{align*}
\bet(T,X)&=\gam^T(X)F^T(X), &     
\calf^T(X)&=X^{-e^T/2}F^T(q^{-(g+1)/2}X).
\end{align*}
If $g$ is even, then $\calf^T\in\QQ[\sqrt{q}][X+X^{-1}]$. 
If $g$ is odd, then $\calf^T\in\QQ\bigl[\sqrt{X},\frac{1}{\sqrt{X}}\bigl]$. 
 
Let $\calc$ be a $g$-dimensional quadratic space over $k$. 
Recall that $S$ is the matrix for the quadratic form on $\calc$ with respect to a fixed basis of $L$, where $L$ is an integral lattice of $\calc$ as explained at the beginning of Section \ref{sec:3}. 
If $g$ is even, $\chi=\chi^\calc$ is unramified and $\det(2S)\in\frko^\times$, then Lemma 14.8 combined with Proposition 14.3 of \cite{Sh2} gives
\beq
\alp\left(S\perp\frac{1}{2}\begin{pmatrix} & \ono_r \\ \ono_r & \end{pmatrix} ,T\right)=\bet(T,\chi(\vpi)q^{-(g+2r)/2}). \label{tag:41}
\eeq  

For the rest of this paper we require $g$ to be even.

\begin{proposition}\label{prop:41}
If $g$ is even, $\chi$ is unramified, $\chi^T=\chi$, $\eta^T=-1$, $\eta^\calc=1$ and $L$ is a self-dual lattice of $\calc$, then
\[\frac{\partial}{\partial s}W_T\Big(f_{\ch\La L^g\Ra}^{(s)}\Big)\Big|_{s=-1/2}
=-\frac{\sqrt{\frkd_k}^g\log q}{\gam(\calc)^g}\frac{\xi^T}{\sqrt{q}^g}\gam^T\biggl(\frac{\xi^T}{\sqrt{q}^g}\biggl)\frac{\partial F^T}{\partial X}\biggl(\frac{\xi^T}{\sqrt{q}^g}\biggl). \] 
\end{proposition}

\begin{proof}
By assumption $\lim_{s\to -1/2}W_T(f_\vph^{(s)})=0$ in view of Lemma \ref{lem:22}(\ref{lem:221}). 
We combine Proposition \ref{prop:31} and (\ref{tag:41}) with Lemmas A.2-A.3  of \cite{K1} to see that 
\begin{align*}
W_T\Big(f_\vph^{(s)}\Big)
&=\gam(\calc)^{-g}\sqrt{\frkd_k}^g\bet\Big(T,\xi^T q^{-(g+1+2s)/2}\Big)\\
&=\gam(\calc)^{-g}\sqrt{\frkd_k}^g\gam^T\Big(\xi^T q^{-(g+1+2s)/2}\Big)F^T\Big(\xi^T q^{-(g+1+2s)/2}\Big). 
\end{align*}
Since $\chi^T=\chi$, we see that $F^T(\xi^T q^{-g/2})=0$. 
We can obtain the stated identity by differentiating this equality at $s=-\frac{1}{2}$. 
\end{proof}

\begin{definition}\label{def:41}
Let $T=(t_{ij})\in\frac{1}{2}\cale_g(\frko)\cap\GL_g(k)$. 
We denote by $S(T)$ the set of all nondecreasing sequences $(a_1,\dots,a_g)$ of nonnegative integers such that $\ord t_{ii}\geq a_i$ and $\ord(2t_{ij})\geq\frac{a_i+a_j}{2}$ for $1\leq i,j\leq g$. 
The Gross--Keating invariant $\GK(T)$ of $T$ is the greatest element of $\bigcup_{U\in\GL_g(\frko)}S(T[U])$ with respect to the lexicographic order. 
\end{definition}

Here, the lexicographic order is defined as follows: 
$(y_1,\dots,y_g)$ is greater than $(z_1,\dots,z_g)$ if there is an integer $1\leq j\leq g$ such that $y_i=z_i$ for $i<j$ and $y_j>z_j$. 
Ikeda and Katsurada \cite{IK2} define a set $\EGK(T)$ of invariants of $T$ attached to $\GK(T)$, which they call the extended Gross--Keating datum of $T$. 
They associated to an extended Gross--Keating datum $H$ a polynomial $\calf^H(Y,X)\in\ZZ[Y^{1/2},Y^{-1/2},X,X^{-1}]$ and show that 
\[\calf^{\EGK(T)}(\sqrt{q},X)=\calf^T(X). \]
When $g$ is even and $\frkd^T=\frko$, one can associate to $\EGK(T)$ truncated extended Gross--Keating datum $\EGK(T)'$ of length $g-1$ by Proposition 4.4 of \cite{IK2}. 
By Definitions 4.2-4.4 of \cite{IK2}  
\begin{align*}
\calf^{\EGK(T)}(Y,X)=&Y^{\frke'/2}X^{-(\frke-\frke'+2)/2}\frac{1-\xi^T Y^{-1}X}{X^{-1}-X}\calf^{\EGK(T)'}(Y,YX)\\
&+Y^{\frke'/2}X^{(\frke-\frke'+2)/2}\frac{1-\xi^T Y^{-1}X^{-1}}{X-X^{-1}}\calf^{\EGK(T)'}(Y,YX^{-1}),  
\end{align*} 
where $\GK(T)=(a_1,\cdots,a_g)$, $\frke=2\left[\frac{a_1+\cdots+a_g}{2}\right]$ and $\frke'=a_1+\cdots+a_{g-1}$. 
It is worth noting that since $\frkd^T=\frko$, we have $\frke=a_1+\cdots+a_g=e^T$. 
We put 
\[F^H(X)=(q^{(g+1)/2}X)^{\frke/2}\calf^H(\sqrt{q}, q^{(g+1)/2}X).\]

If $q$ is odd, then $T$ is equivalent to a diagonal matrix $\diag[t_1,\cdots,t_g]$ with $\ord t_1\leq\cdots\leq \ord t_g$ and the (naive) extended Gross--Keating datum $\EGK(T)=(a_1,\cdots,a_g;\vep_1,\dots,\vep_g)$ is given by 
\begin{align*}
a_i&=\ord t_i, & 
T^{(i)}&=\diag[t_1,\cdots,t_i], & 
\vep_i&=\begin{cases}
\eta^{T^{(i)}} &\text{if $i$ is odd, }\\
\xi^{T^{(i)}} &\text{if $i$ is even }
\end{cases}
\end{align*}
and $\EGK(T)'=(a_1,\cdots,a_{g-1};\vep_1,\dots,\vep_{g-1})$. 

\begin{theorem}\label{thm:41}
Assume that $g$ is even and that $\frkd^T=\frko$. 
Then 
\[F^H(\xi^Tq^{-g/2})=q^{e^T/2}F^{H'}(\xi^Tq^{-g/2}), \] 
where we put $H=\EGK(T)$ and $H'=\EGK(T)'$. 
If $\eta^T=-1$, then
\[\frac{\xi^T}{\sqrt{q}^g}\frac{\partial F^H}{\partial X}\biggl(\frac{\xi^T}{\sqrt{q}^g}\biggl)
=\frac{F^{H'}(\xi^T q^{(2-g)/2})}{q-1}
-\sqrt{q}^{e^T}\frac{\xi^T}{\sqrt{q}^g}\frac{\partial F^{H'}}{\partial X}\biggl(\frac{\xi^T}{\sqrt{q}^g}\biggl). \]
\end{theorem}

\begin{proof}
Substituting $Y=\sqrt{q}$ into $\calf^H(Y,X)$, we get 
\begin{align*}
\calf^H(\sqrt{q},X)
=&X^{-(\frke+2)/2}\frac{1-\xi^Tq^{-1/2}X}{X^{-1}-X}(\sqrt{q}X)^{\frke'/2}\calf^{H'}(\sqrt{q},\sqrt{q}X)\\
&+X^{(\frke+2)/2}\frac{1-\xi^Tq^{-1/2}X^{-1}}{X-X^{-1}}(\sqrt{q}X^{-1})^{\frke'/2}\calf^{H'}(\sqrt{q},\sqrt{q}X^{-1})\\
=&X^{-(e^T+2)/2}\frac{1-\xi^Tq^{-1/2}X}{X^{-1}-X}F^{H'}(q^{(1-g)/2}X)\\
&+X^{(e^T+2)/2}\frac{1-\xi^Tq^{-1/2}X^{-1}}{X-X^{-1}}F^{H'}(q^{(1-g)/2}X^{-1}).  
\end{align*} 
By letting $X=\xi^T\sqrt{q}$, we get 
\[(\xi^T\sqrt{q})^{-e^T/2}F^H(\xi^Tq^{-g/2})=\calf^H(\sqrt{q},\xi^T\sqrt{q})=(\xi^T\sqrt{q})^{e^T/2}F^{H'}(\xi^Tq^{-g/2}). \] 
In the proof of Proposition \ref{prop:41} we have seen that if $\eta^T=-1$, then 
\[\calf^H(\sqrt{q},\xi^T\sqrt{q})=\calf^T(\xi^T\sqrt{q})=(\xi^T\sqrt{q})^{-e^T/2}F^T(\xi^Tq^{-g/2})=0, \]
and hence $F^{H'}(\xi^Tq^{-g/2})=0$. 
We can prove the stated identity by differentiating the equality above at $X=\xi^T\sqrt{q}$. 
\end{proof}

We will use the following result in the next section. 

\begin{lemma}\label{lem:41}
If $T$ is a split symmetric half-integral matrix of size $4$ over $\ZZ_p$, then there exists a nondegenerate isotropic symmetric half-integral matrix $B$ of size $3$ over $\ZZ_p$ such that $F^B_p=F^{\EGK_p(T)'}_p$. 
\end{lemma}

\begin{proof}
If $p=2$, then the existence of such $B$ follows from Proposition 6.4 of \cite{IK} and Theorem 1.1 of \cite{IK2}.
If $p$ is odd, then $T$ is equivalent to a diagonal matrix $\diag[t_1,\cdots,t_4]$ with $\ord t_1\leq\cdots\leq \ord t_4$.
Then we may choose $B$ as $\diag[t_1,\cdots,t_3]$ by using the argument explained in the paragraph just before Theorem \ref{thm:41}.
\end{proof}


\section{The case $g=4$}\label{sec:5}

We discuss the classical Eisenstein series of Siegel. 
For this it is simplest to work over $k=\QQ$. 
Provided that $g$ is a multiple of $4$, we consider the series  
\[E_g(Z,s)=\sum_{\{C,D\}}\det(CZ+D)^{-g/2}|\det(CZ+D)|^{-s}(\det Y)^{s/2}. \]
Here the sum extends over all symmetric coprime pairs modulo $\GL_g(\ZZ)$. 
Let $\calc_p=\calh(\QQ_p)^{g/2}$ be the split quadratic space of dimension $g$ over $\QQ_p$. 
Define $\vph=\otimes_p\vph_p$ by taking $\vph_p=\ch\La\Mat_{g,g}(\ZZ_p)\Ra\in\cals(\calc_p^g)$. 
It is known that $E_g\bigl(Z,s+\frac{1}{2}\bigl)=E(Z,f_\vph^{(s)})$ (see \S \Roman{fou}.2 of \cite{K2}). 
The series is incoherent if and only if $\frac{g}{4}$ is odd due to Lemma \ref{lem:21}. 

Fix a positive definite symmetric half-integral matrix $T$ of size $g$. 
Recall that $\chi_T$ stands for the primitive Dirichlet character corresponding to $\chi^T$. 
The $T$-th Fourier coefficient of $E_g(Z,s)$ is given by 
\[A(T,Y,s)=\frac{a\bigl(T,Y,s-\frac{1}{2}\bigl)L(s,\chi_T)}{\zet\bigl(s+\frac{g}{2}\bigl)\prod_{i=1}^{g/2}\zet(2s+2i-2)}\prod_{p|D_T}F_p^T(p^{-(2s+g)/2}). \]
The $T$-th Fourier coefficient of $\frac{\partial}{\partial s}E_g(Z,s)|_{s=0}$ is given by 
\[C_g(T)=\frac{\partial}{\partial s}A(T,Y,s)|_{s=0}.\]
Recall that $\Dif(T)=\{p\;|\;\eta^T_p=-1\}$. 

\begin{proposition}\label{prop:51}
Assume that $\frac{g}{4}$ is odd. 
Let $T\in\frac{1}{2}\cale_g(\ZZ)\cap\Sym_g^+$. 
\begin{enumerate}
\renewcommand\labelenumi{(\theenumi)}
\item\label{prop:511} If $\chi_T=1$, then $C_g(T)=0$ unless $\Dif(T)$ is a singleton. 
\item\label{prop:512} If $\chi_T=1$ and $\Dif(T)=\{p\}$, then 
\[C_g(T)=-\frac{2^{(g+2)/2}p^{-(g+e^T_p)/2}\log p}{\zet\bigl(1-\frac{g}{2}\bigl)\prod_{i=1}^{(g-2)/2}\zet(1-2i)}\frac{\partial F^T_p}{\partial X}(p^{-g/2})\prod_{p\neq \ell|D_T}\ell^{-e^T_\ell/2}F_\ell^T(\ell^{-g/2}). \]
\item\label{prop:513} If $\chi_T\neq 1$, then  
\[C_g(T)=-\frac{2^{(g+2)/2}L(1,\chi_T)}{\zet\bigl(1-\frac{g}{2}\bigl)\prod_{i=1}^{(g-2)/2}\zet(1-2i)}\prod_{p|D_T}p^{-e^T_p/2}F_p^T(p^{-g/2}). \]
\end{enumerate}
\end{proposition}

\begin{proof}
We have already proved (\ref{prop:511}) in Proposition \ref{prop:21}. 
Taking 
\[\zet(2i)=(-1)^i\frac{(2\pi)^{2i}}{2(2i-1)!}\zet(1-2i)\] 
into account, we have 
\[\zet\biggl(\frac{g}{2}\biggl)\prod_{i=1}^{(g-2)/2}\zet(2i)=\frac{(2\pi)^{g^2/4}\zet\bigl(1-\frac{g}{2}\bigl)}{2^{g/2}\bigl(\frac{g}{2}-1\bigl)!}\prod_{i=1}^{(g-2)/2}\frac{\zet(1-2i)}{(2i-1)!}\]
Recall that $a\bigl(T,Y,-\frac{1}{2}\bigl)=\frac{2^g\pi^{g^2/2}}{\Gam_g(\frac{g}{2})D_T^{1/2}}$ by (\ref{tag:22}).  
Since
\begin{align*}
\vGm_g\biggl(\frac{g}{2}\biggl)&=\frac{\pi^{g^2/4}}{2^{(g^2-2g)/4}}\prod_{i=1}^{(g-2)/2}(2i)!, &
\zet(0)&=-\frac{1}{2}, & 
L'(0,\chi_T)&=\frac{\sqrt{\frkd^T}}{2}L(1,\chi_T), 
\end{align*}
we get (\ref{prop:512}) and (\ref{prop:513}). 
\end{proof}

Hereafter we let $g=4$. 
By a quaternion algebra over a field $k$ we mean a central simple algebra over $k$ of dimension $4$. 
Let $\BB_p$ denote the definite quaternion algebra over $k=\QQ$ that ramifies only at a prime number $p$. 
The reduced norm $\Nr$ on $\BB_p$ defines a positive definite quadratic space $\calv_p$. 
Fix a maximal order $\calo_p$ of $\BB_p$. 
Let $\vph_\ell\in\cals(\calc_\ell^g)$ be the characteristic function of $\Mat_2(\ZZ_\ell)^g$ and $\vph_p'\in\cals(\calv^g_p(\QQ_p))$ the characteristic function of $\calo_p^g\otimes\ZZ_p$. 
We regard $\vph'=\vph'_p\otimes(\otimes_{\ell\neq p}\vph_\ell)$ as the characteristic function of $\calo_p^g\otimes\hat\ZZ$. 
We write $S_p$ for the matrix representation of $\calv_p$ with respect to a $\ZZ$-basis of $\calo_p$. 
Put
\[S_0=\diag\biggl[\begin{pmatrix} 0 & \frac{1}{2} \\ \frac{1}{2} & 0\end{pmatrix}, \begin{pmatrix} 0 & \frac{1}{2} \\ \frac{1}{2} & 0\end{pmatrix}\biggl]. \]

\begin{lemma}\label{lem:51}
Let $T\in\Sym_g(\QQ_p)$. 
\begin{enumerate}
\renewcommand\labelenumi{(\theenumi)}
\item\label{lem:511} If $T\notin\frac{1}{2}\cale_4(\ZZ_p)$, then $W_T\Big(f_{\vph_p}^{(s)}\Big)$ is identically zero. 
\item\label{lem:512} If $T\in\frac{1}{2}\cale_4(\ZZ_p)$ with $\det T\neq 0$, $\chi^T=1$ and $\eta^T_p=-1$, then 
\[\lim_{s\to-1/2}\frac{W_{S_p}\Big(f_{\vph_p'}^{(s)}\Big)}{W_T\Big(f_{\vph_p'}^{(s)}\Big)}\frac{\frac{\partial}{\partial s}W_T\Big(f_{\vph_p}^{(s)}\Big)}{pW_{S_0}\Big(f_{\vph_p}^{(s)}\Big)}=\biggl(p^{-2}\frac{\partial F^{H'}_p}{\partial X}(p^{-2})-\frac{p^{-e^T_p/2}}{p-1}F_p^{H'}(p^{-1})\biggl)\log p, \]
where we put $H'=\EGK_p(T)'$.  
\end{enumerate}
\end{lemma}

\begin{proof}
The first part is trivial. 
Since 
\[\alp_p(S_p,T)=p^{(e^T_p-2)/2}\alp_p(S_p,S_p)\]
by Hilfssatz 17 of \cite{Si1}, 
it follows from Proposition \ref{prop:31} that 
\[\lim_{s\to-1/2}\frac{W_{S_p}\Big(f_{\vph_p'}^{(s)}\Big)}{W_T\Big(f_{\vph_p'}^{(s)}\Big)}=p^{-(e^T_p-2)/2}. \] 
On the other hand, Proposition \ref{prop:41} and Theorem \ref{thm:41} give
\[\lim_{s\to-1/2}\frac{\frac{\partial}{\partial s}W_T\Big(f_{\vph_p}^{(s)}\Big)}{W_{S_0}\Big(f_{\vph_p}^{(s)}\Big)}
=\biggl(p^{(e^T_p-4)/2}\frac{\partial F^{H'}_p}{\partial X}(p^{-2})-\frac{F_p^{H'}(p^{-1})}{p-1}\biggl)\log p. \]
These complete our proof. 
\end{proof}

Let $\bar\FF_p$ be an algebraic closure of a finite field $\FF_p$ with $p$ elements. 
For two supersingular elliptic curves $E,E'$ over $\bar\FF_p$ we consider the free $\ZZ$-module $\Hom(E',E)$ of homomorphisms $E'\to E$ over $\bar\FF_p$ together with the quadratic form given by the degree. 
As $E$ and $E'$ are supersingular, $\Hom(E',E)$ has rank $4$ as a $\ZZ$-module. 
For two quadratic spaces over $\ZZ$ we write $N(L,L')$ for the number of isometries $L'\to L$.  

We are now ready to prove our main result. 

\begin{theorem}\label{thm:51}
If $T\in\frac{1}{2}\cale_4(\ZZ)$ is positive definite, $\chi_T=1$ and $\Dif(T)$ consists of a single prime $p$, then 
\[C_4(T)=2^6\cdot 3^2\biggl(p^{-2}\frac{\partial F^{H'}_p}{\partial X}(p^{-2})-\frac{F_p^{H'}(p^{-1})}{\sqrt{p}^{e^T_p}(p-1)}\biggl)\log p\sum_{(E',E)}\frac{N(\Hom(E',E),T)}{\sharp\Aut(E)\sharp\Aut(E')}, \]
where we put $H'=\EGK_p(T)'$ and where $(E',E)$ extends over all pairs of isomorphism classes of supersingular elliptic curves over $\bar\FF_p$. 
\end{theorem}

\begin{proof}
Proposition \ref{prop:33} and (\ref{tag:32}) applied to $L=\calo_p$ gives 
\begin{align*}
C_4(T)
&=R'(\calo_p,T)c\lim_{s\to -1/2}\frac{W_{S_p}\Big(f_{\vph_p'}^{(s)}\Big)}{W_T\Big(f_{\vph_p'}^{(s)}\Big)}\frac{\frac{\partial}{\partial s}W_T\Big(f_{\vph_p}^{(s)}\Big)}{pW_{S_0}\Big(f_{\vph_p}^{(s)}\Big)},  
\end{align*}
where 
\[c=\frac{2p}{\frkm'(\calo_p)}\lim_{s\to-1/2}\frac{W_{S_0}\Big(f_{\vph_p}^{(s)}\Big)}{W_{S_p}\Big(f_{\vph_p'}^{(s)}\Big)}.\]
If $T=S_p$, then we claim that $R'(\calo_p,S_p)=1$.
To prove this, it suffices to show that 
$N(\scrl,S_p)=0$ if $\scrl$ is not isometric to $\calo_p$
and $N(\calo_p,S_p)=\sharp\O(\calo_p)$, where $\scrl \in \Xi'(\calo_p)$.
If $N(\scrl, S_p)\neq 0$, then there is an injection $f:\calo_p \rightarrow \scrl$ as a lattice preserving the associated quadratic forms. 
Thus we only need to show that $f$ is surjective.
If it is not surjective, then $\scrl$ and $\calo_p$ have different discriminant, which is a contradiction to the assumption that $\scrl$ and $\calo_p$ are in the same genus.

Applying Proposition \ref{prop:32} and (\ref{tag:32}) to $T=S_p$, we get  
\[\frac{2}{\frkm'(\calo_p)}=c_4D_{S_p}^{-1/2}\lim_{s\to-1/2}W_{S_p}\Big(f_{\vph_p'}^{(s)}\Big)\prod_{\ell\neq p}W_{S_p}\Big(f_{\vph_\ell}^{(s)}\Big). \]
It follows that 
\begin{align*}
c
&=pc_4D_{S_p}^{-1/2}\lim_{s\to-1/2}\prod_\ell W_{S_0}\Big(f_{\vph_\ell}^{(s)}\Big)\\
&=c_4\lim_{s\to-1/2}\prod_\ell \gam^S_\ell(\ell^{-(5+2s)/2})=\frac{c_4}{\zet(2)^2}\lim_{s\to-1/2}\frac{\zet\bigl(s+\frac{1}{2}\bigl)}{\zet(2s+1)}=2^7\cdot 3^2. 
\end{align*}
Since 
$R(\calo_p,T)=2R'(\calo_p,T)$ by (\ref{tag:31}) and (\ref{tag:32}), and 
\beq
R(\calo_p,T)=\sum_{\scrl\in\Xi(\calo_p)}\frac{N(\scrl,T)}{\sharp\SO(\scrl)}=\sum_{(E',E)}\frac{N(\Hom(E',E),T)}{\sharp\Aut(E)\sharp\Aut(E')} \label{tag:51}
\eeq
by Proposition 4.1 of \cite{W1}, our statement follows from Lemma \ref{lem:51}(\ref{lem:512}). 
\end{proof}

\begin{conjecture}\label{conj:51}
Let $\calv$ be a totally positive definite quadratic space over a totally real number field $k$ of dimension $g$. 
Fix a maximal integral lattice $L$ of $\calv$. 
Let $T\in\frac{1}{2}\cale_g(\frko)$ be totally positive definite. 
If $g$ is even and $\chi^\calv=1$, then there is a totally positive definite matrix $T'\in\frac{1}{2}\cale_{g-1}(\frko)$ such that 
\[R(L,T)=2R(L,T'). \]
\end{conjecture}

\begin{proposition}\label{prop:52}
If $k=\QQ$ and $g=4$, then Conjecture \ref{conj:51} is true. 
\end{proposition}

\begin{proof}
Since $R(L,T)=0$ unless $\Dif(T)=\Dif(\calv)$,  we may assume that 
\[\Dif(T)=\Dif(\calv). \]
Lemma \ref{lem:41} gives $T'_p\in\frac{1}{2}\cale_3(\ZZ_p)$ such that $F^{T'_p}_p=F^{\EGK_p(T)'}_p$ for every rational prime $p$. 
In addition, the proof of Lemma \ref{lem:41} yields that $T'_p$ is unimodular for almost all primes $p$.
Thus we can find a positive rational number $0<\del\in \QQ^\times$ such that $\del^{-1}\det T'_p\in\ZZ_p^\times$ for every $p\notin\Dif(\calv)$. 
For $p\in\Dif(\calv)$ we fix an arbitrary anisotropic ternary quadratic form $T_p'$ over $\ZZ_p$. 
Recall that $\alp_p(S_p,T_p')$ is independent of the choice of $T'_p$.  

Since $F^{uT'_p}_p=F^{T'_p}_p$ for $u\in\ZZ_p^\times$, there is no harm in assuming that $\del=\det T'_p$. 
Since $\eta^{T'_p}_p=1$ for $p\notin\Dif(\calv)$, the Minkowski-Hasse theorem gives $z\in\Sym_3(\QQ)$ which is positive definite and such that $z\in T'_p[\GL_3(\QQ_p)]$ for every $p$. 
Take $A\in\GL_3(\AAf)$ so that $z=T'_p[A_p]$ for every $p$. 
We can take $D\in\GL_3(\QQ)$ in such a way that $AD^{-1}\in\GL_3(\ZZ_p)$ for every $p$. 
Put $T'=z[D^{-1}]$. 
Then $T'\in T'_p[\GL_3(\ZZ_p)]$ for every $p$. 
In particular, $T'\in\frac{1}{2}\cale_3(\ZZ)$. 

In view of (\ref{tag:32}) it suffices to show that  
\[\frac{R'(L,T)}{\frkm'(L)}=2\frac{R'(L,T')}{\frkm'(L)}. \] 
We see by the Siegel formula that 
\[\frac{R'(L,T)}{\frkm'(L)}
=2^{-1}d_\infty(L,T)2^4\prod_{p\in\Dif(\calv)} \frac{\alp_p(S_p,T)}{2}\prod_{q\notin\Dif(\calv)}(1-q^{-2})^2F_q^T(q^{-2}). \]
Recall that the archimedean densities are given by 
\begin{align*}
d_\infty(L,T)&=\frac{\prod_{i=1}^4\frac{\pi^{i/2}}{\Gam\left(\frac{i}{2}\right)}}{\det(2T)^{1/2}[L^*:L]^2}, & 
d_\infty(L,T')&=\frac{\prod_{i=2}^4\frac{\pi^{i/2}}{\Gam\left(\frac{i}{2}\right)}}{[L^*:L]^{3/2}}. 
\end{align*}
Since 
\begin{align*}
\alp_p(S_p,T')&=2(p+1)(1+p^{-1}), & 
\alp_p(S_p,T)=4p^{e^T_p/2}(p+1)^2. 
\end{align*}
by \cite[Theorem 1.1]{W2} and Proposition 6.5 of \cite{CY}. 
The latter result can be derived more generally from Shimura's exact mass formula. 
Since $[L^*:L]=\prod_{p\in\Dif(\calv)}p^2$ by assumption, we have 
\[d_\infty(L,T)=[L^*:L]^{-2}\det(2T)^{-1/2}\prod_{i=1}^4\frac{\pi^{i/2}}{\Gam\left(\frac{i}{2}\right)}=\frac{d_\infty(L,T')}{\det(2T)^{1/2}}\prod_{p\in\Dif(\calv)}p^{-1}. \]
We combine these with Theorem \ref{thm:41} to obtain
\[\frac{R'(L,T)}{\frkm'(L)}
=d_\infty(L,T')2^3\prod_{p\in\Dif(\calv)}\alp_p(S_p,T')\prod_{q\notin\Dif(\calv)}(1-q^{-2})^2F_q^{T'}(q^{-2}). \]
The final expression equals $2\frac{R'(L,T')}{\frkm'(L)}$ by the Siegel formula. 
\end{proof}

\begin{corollary}\label{cor:51}
If $T$ is a positive definite symmetric half-integral matrix of size $4$ which satisfies $\chi^T=1$ and $\eta^T_\ell=1$ for $\ell\neq p$, then there exists a positive definite symmetric half-integral matrix $T'$ of size $3$ such that 
\[\sum_{(E',E)}\frac{N(\Hom(E',E),T)}{\sharp\Aut(E)\sharp\Aut(E')}=2\sum_{(E',E)}\frac{N(\Hom(E',E),T')}{\sharp\Aut(E)\sharp\Aut(E')}, \]
where $(E,E')$ extends over all pairs of isomorphism classes of supersingular elliptic curves over $\bar\FF_p$. 
\end{corollary}

\begin{proof}
Proposition 4.1 of \cite{W1} gives 
\[R(\calo_p,T')=\sum_{L\in\Xi(\calo_p)}\frac{N(L,T')}{\sharp\SO(L)}=\sum_{(E',E)}\frac{N(\Hom(E',E),T')}{\sharp\Aut(E)\sharp\Aut(E')}. \]
We can derive Corollary \ref{cor:51} from (\ref{tag:51}) and Proposition \ref{prop:52}. 
\end{proof}

Let $T\in\frac{1}{2}\cale_4(\ZZ_p)$ be an anisotropic symmetric matrix with (naive) extended Gross-Keating invariant $(a_1,a_2,a_3,a_4;\vep_1,\vep_2,\vep_3,\vep_4)$. 
Note that $\vep_1=\vep_4=1$ by definition. 
One can easily see that $\vep_2\neq 1$ and $\vep_3=-1$. 
Proposition 5.3 of \cite{CY} gives a partition $\{1,2,3,4\}=\{i,j\}\cup\{k,l\}$ such that 
\[a_i\equiv a_j\not\equiv a_k\equiv a_l\pmod 2. \]

\begin{lemma}\label{lem:53}
\begin{enumerate}
\renewcommand\labelenumi{(\theenumi)}
\item\label{lem:531} If $a_1\not\equiv a_2\pmod 2$, then 
\begin{align*}
F_p^{T'}(p^{-1})
=&\frac{p^{a_1+1}-1}{(p-1)(p^3-1)}\biggl(p^{\{a_1+3(a_2+1)\}/2}-\frac{p^{a_1+1}+1}{p+1}\biggl)\\
&-\frac{p^{(a_1+a_2+2a_3+1)/2}}{p-1}\biggl\{(a_1+1)p^{(a_1+a_2+1)/2}-\frac{p^{a_1+1}-1}{p-1}\biggl\}. 
\end{align*}
\item\label{lem:532} If $a_1\equiv a_2\pmod 2$, then \begin{align*}
F_p^{T'}(p^{-1})
=&\frac{p^{a_1+1}-1}{(p-1)(p^3-1)}\biggl(p^{(a_1+3a_2)/2}-\frac{p^{a_1+1}+1}{p+1}\biggl)\\
&-\frac{p^{(a_1+a_2+2a_3+2)/2}}{p-1}\biggl\{(a_1+1)p^{(a_1+a_2)/2}-\frac{p^{a_1+1}-1}{p-1}\biggl\}\\
&+p^{(a_1+3a_2)/2}\frac{p^{a_1+1}-1}{p^2-1}(p^{a_1-a_2+1}+1). 
\end{align*}\end{enumerate}
\end{lemma}

\begin{proof}
We write the naive extended Gross-Keating invariant of $T$ as 
\[\EGK_p(T)=(a_1,a_2,a_3,a_4;1,\vep_2,\vep_3,1). \] 
Let $\sig$ be either $1$ or $2$ according as $a_1-a_2$ is odd or even.    
Section 8 of \cite{IK2} expresses $F^{\EGK_p(T)'}_p(X)$ in terms of $\EGK_p(T)'=(a_1,a_2,a_3;1,\vep_2,\vep_3)$: 
\begin{multline*}
F_p^{\EGK_p(T)'}(p^{-2}X)=\sum_{i=0}^{a_1}\sum_{j=0}^{(a_1+a_2-\sig)/2-i}p^{i+j}X^{i+2j}\\
\vep_3\sum_{i=0}^{a_1}\sum_{j=0}^{(a_1+a_2-\sig)/2-i}p^{(a_1+a_2-\sig)/2-j}X^{a_3+\sig+i+2j}\\
+\vep_2^2p^{(a_1+a_2-\sig+2)/2}\sum_{i=0}^{a_1}\sum_{j=0}^{a_3-a_2+2\sig-4}\vep_2^jX^{a_2-\sig+2+i+j}. 
\end{multline*}
We now specialize the formula to $X=p$ and $\vep_3=-1$. 
Then 
\begin{align*}
F_p^{T'}(p^{-1})
=&\frac{p^{a_1+1}-1}{(p-1)(p^3-1)}\biggl(p^{\{a_1+3(a_2-\sig+2)\}/2}-\frac{p^{a_1+1}+1}{p+1}\biggl)\\
&-\frac{p^{(a_1+a_2+2a_3+\sig)/2}}{p-1}\biggl((a_1+1)p^{(a_1+a_2-\sig+2)/2}-\frac{p^{a_1+1}-1}{p-1}\biggl)\\
&+\vep_2^2p^{\{a_1+3(a_2-\sig+2)\}/2}\frac{(p^{a_1+1}-1)(1-(\vep_2p)^{a_1-a_2+2\sig-3})}{(p-1)(1-\vep_2p)}. 
\end{align*}
Since $\vep_2=0$ or $-1$ according as $a_1-a_2$ is odd or even by Proposition 2.2 of \cite{IK} and Proposition 5.4 of \cite{CY}, we obtain the stated formulas.  
\end{proof}

The degree $\deg\scrz(B)$ is defined in (\ref{tag:12}) for positive definite symmetric half-integral $3\times 3$ matrices $B$ such that $\Dif(B)$ is a singleton. 

\begin{corollary}\label{cor:52}
Let $T$ be a positive definite symmetric half-integral $4\times 4$ matrix such that $\chi_T=1$ and $\Dif(T)=\{p\}$. 
Let $\sig$ be either $1$ or $2$ according as $a_1-a_2$ is odd or even. 
If $\deg\scrz(T')\neq 0$, then 
\[\biggl|\frac{C_4(T)}{-2^8\cdot 3^2\cdot \deg\scrz(T')}-1\biggl|<\frac{4}{p\sqrt{p}}\biggl(p^{-(a_4-3+\sig)/2}+\frac{4p^{-(a_4-a_1)/2}}{a_1+1}\biggl), \]
where $\GK_p(T)=(a_1,a_2,a_3,a_4)$. 
In particular, 
\begin{align*}
\biggl|\frac{C_4(T)}{-2^8\cdot 3^2\cdot \deg\scrz(T')}-1\biggl|&<\frac{20}{p\sqrt{p}}, & 
\lim_{e_p^T\to\infty}\frac{C_4(T)}{-2^9\cdot 3^2\cdot \deg\scrz(T')}&=1. 
\end{align*}
\end{corollary} 

\begin{proof}
By (2.12) and (2.13) of \cite{W2}
\begin{align*}
-p^{-2}\frac{\partial F^{H'}_p}{\partial X}(p^{-2})&\geq (a_1+1)p^{(a_1+a_2)/2}\biggl(\frac{a_3-a_2+2\sig}{\sqrt{p}^\sig}+\vep_2^2\frac{a_3-a_2+1}{2}\biggl)\\
&\geq (a_1+1)p^{(a_1+a_2-(2-\sig))/2}. 
\end{align*}
Recall that if $\sig=1$, then $a_1<a_2\leq a_3\leq a_4$ while if $\sig=2$, then $a_1\leq a_2<a_3\leq a_4$. 
An examination of the proof of Lemma \ref{lem:53} confirms that 
\begin{multline*}
\biggl|\frac{F_p^{H'}(p^{-1})}{\sqrt{p}^{e^T_p}(p-1)}\biggl|\leq\frac{a_1+1}{(p-1)^2}p^{(a_1+a_2-a_4+2)/2}+\frac{p^{a_1+a_2-(a_3+a_4+3\sig)/2+4}}{(p-1)^2(p^3-1)}\\
+\vep_2^2\frac{p^{2a_1+2-(a_3+a_4)/2}}{(p-1)^2(p+1)}+\vep_2^2\frac{p^{a_1+a_2+1-(a_3+a_4)/2}}{(p-1)^2(p+1)}\\
<4p^{(a_1+a_2)/2-1}\{(a_1+1)p^{-a_4/2}+2p^{-(a_4-a_1+3\sig)/2}+2\vep_2^2p^{-(a_4-a_1+1)/2}\}. 
\end{multline*}
Now our proof is completed by Theorem \ref{thm:12}. 
\end{proof}



\end{document}